\documentclass[11pt]{amsart}
\usepackage{times}
\usepackage[T1]{fontenc}
\usepackage[utf8]{inputenc}
\usepackage{lmodern}
\usepackage{color}
\usepackage[english]{babel}
\usepackage{geometry}
\usepackage{relsize}
\usepackage{booktabs}
\usepackage{mathtools}
\usepackage{amsthm}
\usepackage{amsmath,amssymb}
\usepackage{emptypage}
\usepackage{xfrac}    
\usepackage{faktor}
\usepackage{graphicx}
\usepackage{listings}
\usepackage{lipsum}
\usepackage{setspace}
\usepackage{lipsum}
\usepackage{newlfont}
\usepackage[mathscr]{eucal}	
\usepackage{physics}
\usepackage{tikz}
\usepackage{url}
\usepackage{hyperref}
\usetikzlibrary{positioning}
\usepackage{xcolor}
\usepackage[new]{old-arrows}
\usepackage{multirow}
\usepackage{enumitem}
\usepackage[toc]{appendix}

\def\eqref#1{(\ref{#1})}
\newlength{\alignedproofindent}
\NewDocumentEnvironment{alignedproof}{O{\proofname}m}
 {%
  \proof[#1]\mbox{}\par\nopagebreak
  \settowidth{\alignedproofindent}{#2\quad}%
  \everypar{\hangindent=\alignedproofindent\hangafter=0 }%
 }
 {\endproof}
\NewDocumentCommand{\alignedproofstep}{m}{%
  \par\noindent\makebox[0pt][r]{%
    \makebox[\alignedproofindent][l]{#1}%
  }\ignorespaces
}
\newcommand{\gives}{${{}\Rightarrow{}}$}

\theoremstyle{plain}
\newtheorem{theorem}{Theorem}[section]
\newtheorem{lemma}[theorem]{Lemma}
\newtheorem{proposition}[theorem]{Proposition}
\newtheorem{corollary}[theorem]{Corollary}

\theoremstyle{definition}
\newtheorem{definition}[theorem]{Definition}

\newtheorem{question}[theorem]{Question}
\newtheorem{remark}{Remark}[section]
\usepackage{setspace}
\usepackage[signatures,write]{frontespizio}

\usepackage{fancyhdr}

\newcommand{\R}{\mathbb{R}} % \R   = numeri reali
\newcommand{\C}{\mathbb{C}} % \C   = numeri complessi
\AtBeginDocument
{
\setlength\abovedisplayskip{.6\baselineskip}
\setlength\abovedisplayshortskip{.6\baselineskip}
\setlength\belowdisplayskip{.5\baselineskip}
\setlength\belowdisplayshortskip{.5\baselineskip}
}
\begin{document}

\title{On the structure of compact strong HKT manifolds}

\author{Beatrice Brienza}
\address[Beatrice Brienza]{Dipartimento di Matematica ``G. Peano'', Universit\`{a} degli studi di Torino \\
Via Carlo Alberto 10\\
10123 Torino, Italy}
\email{beatrice.brienza@unito.it}

\author{Anna Fino}
\address[Anna Fino]{Dipartimento di Matematica ``G. Peano'', Universit\`{a} degli studi di Torino \\
Via Carlo Alberto 10\\
10123 Torino, Italy\\
\& Department of Mathematics and Statistics, Florida International University\\
Miami, FL 33199, United States}
\email{annamaria.fino@unito.it, afino@fiu.edu}

\author{Gueo Grantcharov}
\address[Gueo Grantcharov]{Department of Mathematics and Statistics \\
Florida International University\\
Miami, FL 33199, United States}
\email{grantchg@fiu.edu}

\author{Misha Verbitsky}
\address[Misha Verbitsky]{Instituto Nacional de Matem\'atica Pura e Aplicada (IMPA)\\
Estrada Dona Castorina, 110\\
Jardim Bot\^anico, CEP 22460-320 \\
Rio de Janeiro, RJ - Brasil}
\email{verbit@impa.br}



\maketitle

\begin{abstract} We study the geometry of compact strong HKT and, more generally, compact BHE manifolds. We prove that any compact BHE manifold with full holonomy must be Kähler and we establish a similar result for strong HKT manifolds. Additionally, we demonstrate a rigidity theorem for strong HKT structures on solvmanifolds and we completely classify those with parallel Bismut torsion. Finally, we introduce the Ricci foliation for hypercomplex manifolds and analyze its properties for compact, simply connected, 8-dimensional strong HKT manifolds, proving that they are always Hopf fibrations over a compact $4$-dimensional orbifold. 
\end{abstract}

\section{introduction}
K\"ahler geometry lies in the intersection of three fundamental geometries: complex, symplectic, and Riemannian. It also naturally appeared in the early days of String Theory \cite{CZ-84,HS-84,GHR-84,BCZ-85,St, Hull, HP-88, HP-96,  GuP2, GP-04, IP, Gr-Poon,GGP, FIUV}. In particular, given a Hermitian manifold $(M,J,g)$, namely a complex manifold endowed with a fixed Riemannian metric $g$ satisfying $g(J\cdot, J\cdot)=g(\cdot,\cdot),$ we say that the Hermitian manifold $(M,J,g)$ is \emph{K\"ahler} if the associated fundamental form $\omega=g \circ J$ is closed, i.e., $d\omega=0$. \par
From the complex point of view, if a Hermitian structure is K\"ahler there exists a further compatibility between the Riemannian and the complex geometries, as the Levi--Civita connection preserves the complex structure $J$, namely $\nabla^{LC} J=0,$ which is not true when $d\omega \neq 0$. Therefore, in the non-K\"ahler setting the Levi--Civita connection is not sufficient to encode the Hermitian geometry of the manifold, making it necessary to seek alternative connections that preserve both $g$ and $J$. These connections are known as \emph{Hermitian connections}. See Section 1 for further details.
The Bismut  (or Strominger) connection \cite{Bi,St} has been introduced as the unique Hermitian connection with totally skew symmetric torsion and its expression is given by
\[
g( \nabla^{B}_X Y, Z) =  g(\nabla^{LC}_X Y, Z) -\frac{1}{2}d^c \omega(X,Y, Z),
\]
for any vector fields $X,Y, Z$. \par
The $3$-form  $H \coloneqq -d^c \omega$ is often called the \emph{Bismut torsion}. The Bismut connection was first introduced in the physics context in the Strominger foundational paper \cite{St}. In fact, it naturally arises in string theory, particularly in heterotic supergravity, where it plays a fundamental role in anomaly cancelation and flux compactifications. \par
When the Bismut torsion is closed, namely when $dd^c\omega=0,$ the Hermitian structure $(J,g)$ is said to be \emph{strong K\"ahler with torsion} (SKT) or \emph{pluriclosed}. \par

As remarked above, the Bismut connection is Hermitian, namely $\nabla^B J= \nabla^B g=0$. As a result, its holonomy is contained in the unitary group $\operatorname{U}(n)$, where $n$  denotes the complex dimension of $M$. When the restricted holonomy of $\nabla^B$ is in the special unitary group $\operatorname{SU}(n)$, the Hermitian structure $(J,g)$ is said \emph{Calabi Yau with torsion} (CYT, in short). Moreover, even in the K\"ahler case, CYT manifolds are not Calabi-Yau in the classical sense. For example, the Ricci flat K\"ahler metric on Enriques surface has restrictel holonomy in $SU(2)$, but it is not Calabi-Yau.  Equivalently, we say that an Hermitian structure $(J, g)$ is CYT if 
\begin{equation} \label{eqn:rhob}
\rho^B(X,Y)=\frac{1}{2} \sum_{i=1}^{2n} R^B(X,Y,Je_i,e_i) =0,  \ \forall X,Y \in \Gamma(\operatorname{TM}),
\end{equation}
where $R^B$ is the Bismut curvature tensor and $\{e_i\}$ is a local orthonormal frame. \par
Since $\rho^B$ is a representative of the first Chern class, $c_1(M)$ represents an obstruction to the existence of CYT metrics. In other words, the condition $\rho^B=0$ implies that the first Chern class must vanish. Despite this topological restriction, CYT metrics are not difficult to construct. For instance, in complex dimension $3$, the compact manifolds $(k-1)(S^2 \times S^4)\sharp k(S^3 \times S^3)$ admit CYT metrics for any integer $k \ge 1$ \cite{GGP}. The existence of CYT metrics has also been investigated on $C$-spaces \cite{GG}. \\
A CYT structure $(J,g)$ is called \emph{Bismut-Hermitian Einstein} (BHE) if it is also SKT. Inspired by the Calabi-Yau Theorem \cite{CAL,CAL2,YAU,YAU2},  a natural initial question is whether the condition $c_1(M) = 0$ guarantees the existence of a Bismut Hermitian-Einstein metric. However, this is not the case, as demonstrated in \cite{GFJS}. Therefore, unlike the K\"ahler Ricci flat case, having vanishing first Chern class is not enough to ensure the existence of BHE metrics. Actually, BHE metrics are rare. The few explicit examples we have suggest that, a rigidity result may be proved implying that these metrics are highly constrained, if non K\"ahler. \par
Since any BHE manifold is CYT, the restricted holonomy of the Bismut connection is contained within the special unitary group $\operatorname{SU}(n)$.
Compact Bismut flat and K\"ahler Ricci flat manifolds are the building blocks of any known example of compact BHE manifolds (see Section 1 for further details), and it has been shown in \cite{BFG, BPT} that a compact BHE manifold has parallel Bismut torsion if, and only if, it admits a finite cover which splits as a product of a compact K\"ahler Ricci flat manifold and a compact Bismut flat one. \par
In \cite{BFG, W}, the authors asked the following challenging question
\begin{question} \label{qst:3}
Is it true that in the compact case the BHE condition forces the parallelism of the Bismut torsion?
\end{question}
In complex dimension $2$ this is true by \cite{GI}.  Compact BHE threefolds have been studied in \cite{ALS}. However, compact examples of BHE threefolds with non parallel Bismut torsion have been constructed recently in \cite{ALL-26}. 

BHE metrics naturally arise in hyper-Hermitian geometry. A \emph{hyper-Hermitian} structure on a $4n$-dimensional manifold is a quadruple $(I,J,K,g)$,  where $I,J,K$ are complex structures satisfying the imaginary quaternionic relations and $g$ is a Riemannian metric compatible with each of them. For each Hermitian structure we denote the corresponding Bismut connection by $\nabla^B_L$, $L=I,J,K$. When the three Bismut connections coincide, i.e., $\nabla^B_I=\nabla^B_J=\nabla^B_K=:\nabla^B$, or equivalently, when the three Hermitian structures share the same Bismut torsion $H$, the hyper-Hermitian structure $(I,J,K,g)$ is called \emph{hyper-K\"ahler with torsion} (HKT, in short).  If, in addition, the torsion 3-form is closed, the structure $(I, J, K, g)$  is called \emph{strong HKT}.
For further details about HKT geometry we refer to \cite{HP, HP-96, GPS,  Gr-Poon, Agr, IP2,OPS, P-25}. \par
In the HKT case, due to the symmetries of the Bismut connection, the holonomy of the Bismut connection lies in the symplectic group $\operatorname{Sp}(n)$, which is a subgroup of $\operatorname{SU}(2n)$. This implies that whenever the structure is strong HKT each Hermitian structure $(L,g)$ is indeed BHE, with $L=I,J,K$. However, very little is currently known about strong HKT manifolds.\par
(Strong) HKT structures  first emerged in string theory as the geometric structures on the target space of a (0,4)-supersymmetric sigma model with a Wess-Zumino term (see \cite{HP, GPS}). They also appear in superconformal quantum mechanics \cite{MiS} and in the geometry of black hole moduli spaces \cite{GuP}.  The existence of strong HKT structures has been investigated in the locally homogeneous setting. In particular, combining the results of \cite{BDV,DF},  a nilmanifold admits an invariant strong HKT structure if, and only if, it is hyper-K\"ahler, and an analogous result for almost abelian solvmanifolds was proved in \cite{AB}. These results prompted the following question \cite{BF}:
\begin{question}  \label{qst:2}
Is it true that the existence of an invariant strong HKT structure on a solvmanifold implies that the manifold is hyper-K\"ahler?
\end{question}
In Section 2, we address Questions \ref{qst:3} and \ref{qst:2}. Specifically, we prove that any compact BHE manifold with full holonomy of the Bismut connection must be Kähler. The key idea is that in the non-Kähler case, the restricted holonomy of the Bismut connection reduces further to $\operatorname{SU}(n-1)$. In fact, every compact BHE manifold admits a unique (up to a constant) smooth function $f$ called the \emph{potential function}, such that the vector field $V=\theta^\sharp-\operatorname{grad} f$ is Bismut-parallel and never vanishes \cite{SU, Le2, ALS,GFJS}.  As elsewhere, $\theta^\sharp$ is the vector field associated to the Lee form $\theta$ via $g$. This result extends to compact strong HKT manifolds by leveraging the fact that each Hermitian structure induced by the strong HKT one is BHE and that their Lee forms coincide. In particular, we show that the vector fields $V_L=\theta^\sharp-\operatorname{grad} f_L$ associated with the three Hermitian structures are identical, further reducing the holonomy of the Bismut connection from $\operatorname{Sp}(n)$ to $\operatorname{Sp}(n-1)$. Finally, we prove that strong HKT solvmanifolds are hyper-K\"ahler by establishing that the Lee form is parallel. i.e. the potential $f$ is constant, and using general curvature properties of left-invariant metrics on solvable Lie algebras.\par
In Section \ref{section3}, we fully characterize strong HKT manifolds with parallel Bismut torsion, extending the results of \cite{BFG, BPT} to the hyper-Hermitian setting. In particular, we establish a Beauville-Bogomolov-type decomposition, proving that any compact strong HKT manifold with parallel Bismut torsion is, up to a finite cover, the product of a compact hyper-Kähler manifold and a Bismut-flat one. \par
In Section \ref{section4}, we introduce the Ricci foliation of a hyper-complex manifold as the kernel of the Obata Ricci curvature and in Section \ref{section5}, we analyze this foliation in depth on compact, simply connected strong HKT $8$-manifolds. We prove that in this case, the foliation admits an integrable sub-distribution isomorphic to $\mathfrak{su}(2)\oplus \mathfrak{u}(1)$, corresponding to an isometric action of $\mathbb{H}^*$. In particular, we show that the potential function $f$ must be constant, leading to a highly rigid geometric structure. The HKT geometry admitting such $\mathbb{H}^*$-action is considered in \cite{Poon-Swann} and recently in \cite{Papadopoulos, Pap-Witten, W}. \par
Using the results in Section \ref{section5} about the Ricci foliation, in Section \ref{section6} we establish an existence of a foliation with compact leaves, induced by an isometric $S^1 \times SU(2)$-action. This result allow us to prove a structure theorem for compact simply connected strong HKT manifolds of dimension 8. \par
In Section \ref{section7} we fully characterize the conditions under which the Bismut torsion is parallel, which, by the splitting result proved in Section \ref{section3}, implies that the manifold must be  the Lie group $SU(3)$, while in the last section we discuss some relations with the HKT potential.
\\ \ \\
\textbf{Acknowledgments:} The authors would like to  thank  Marco Radeschi and Jeﬀrey Streets for useful discussions. The authors are  also grateful to the referees for the valuable comments and suggestions  which improved a lot  the readability of the paper.  Beatrice Brienza and Anna Fino are partially supported by Project PRIN 2022 \lq \lq Real and complex manifolds: Geometry and Holomorphic Dynamics”, by GNSAGA (Indam). Beatrice Brienza is partially supported by the INdAM - GNSAGA Project  CUP E53C24001950001. Anna Fino is also supported by  a grant from the Simons Foundation (\#944448).  Gueo Grantcharov is partially supported by a grant from
the Simons Foundation (\#853269). Misha Verbitsky is partially supported by FAPERJ SEI-260003/000410/2023 and CNPq - Process 310952/2021-2.

\section{Some remarks on the Bismut connection of strong HKT manifolds} \label{section2}
Let  $M^n$ be a smooth $n$-dimensional manifold and let $\nabla$ be a linear connection. Fix any piece-wise smooth loop $\gamma$ based at $p \in M$. The connection $\nabla$ induces the parallel transport map $P_\gamma: \operatorname{T_p M} \to \operatorname{T_p M}$, which is linear and invertible. The holonomy group at $p \in M$ is defined as
\[
\operatorname{Hol}_p(\nabla):=\{P_\gamma \in \operatorname{GL}(\operatorname{T_p M}) \ | \ \gamma \ \text{is a loop based at} \ p \}.
\]
For loops $\gamma$ and $\delta$ based at $p$, we denote by $\gamma \star \delta$ their composition. $\operatorname{Hol}_p(\nabla)$ has a natural structure of a group, i.e., $P_\gamma \circ P_\delta =P_{\gamma \star \delta}$, and it is a subgroup of $\operatorname{GL}(\operatorname{T_p M})$, which is identified with $\operatorname{GL}(n,\R)$ once a basis is fixed. \par
If we require $M$ to be connected, as we will do in what follows, then the holonomy group $\operatorname{Hol}_p(\nabla)$ is independent of the base point $p$. Because of this, the subscript $p$ may be omitted. \par
$\operatorname{Hol}(\nabla)$ needs not to be connected and its identity component $\operatorname{Hol}^{0}_p(\nabla)$ is called the \emph{restricted} holonomy group: it coincides with the parallel transport maps $P_\gamma$ such that $\gamma$ is null-homotopic. Clearly, when $M$ is simply connected, $\operatorname{Hol}(\nabla)=\operatorname{Hol}^0(\nabla)$. \par
Let  $(M,J,g)$ be a Hermitian manifold of complex dimension $n$ and let $\nabla$ be a \emph{Hermitian} connection, i.e., $\nabla$ is a connection on $\operatorname{TM}$ such that $\nabla J=0$ and $\nabla g=0$. Since both the complex structure and the Riemannian metric are parallel, the parallel transport map $P_\gamma$ is an isometry of $(\operatorname{T_pM},g_p)$, for every point $p \in M$,  and satisfies $P_\gamma J=J P_\gamma$. Therefore,
\[
\operatorname{Hol}(\nabla) \subseteq \operatorname{O}(n) \cap \operatorname{GL}(n,\C)=\operatorname{U}(n).
\]
In \cite{Ga}, an affine line of Hermitian connections on the tangent bundle $\operatorname{TM}$ is introduced. They are known as \emph{Gauduchon} or \emph{canonical} connections and the expression of the line is given by
\begin{equation} \label{eqn:canonical}
g(\nabla^t_XY,Z)=g(\nabla^{LC}_XY,Z)+\frac{t-1}{4}(d^c\omega)(X,Y,Z)+\frac{t+1}{4}(d^c\omega)(X,JY,JZ),
\end{equation}
where $d^c \omega=Jd\omega$. We will adopt the convention $Jd\omega(X,Y,Z):=d\omega(-JX,-JY,-JZ)$. \par
When $(M,J,g)$ is K\"ahler, $d^c\omega$ vanishes, and so that the line collapses to a single point, namely, the Levi-Civita connection. Nonetheless, when $(M,J,g)$ is not K\"ahler, the line is not trivial and the connections $\nabla^t$ have non vanishing torsion. \par
For particular values of $t \in \R$, the \emph{Chern} and the \emph{Bismut} connections are recovered, i.e., $\nabla^{1}= \nabla^{Ch}$ and $\nabla^{-1}= \nabla^B$ \cite{Bi, Ch}. Despite $\nabla^{LC},  \nabla^{B}$, and $\nabla^{Ch}$ being pairwise distinct connections, any two of them completely determines the third.  \par
In this work, we will focus mainly on the Bismut connection $\nabla^B$, also known as the \emph{Strominger} connection \cite{Bi, St}.  The Bismut connection can be characterized as the only Hermitian connection with totally skew-symmetric torsion, and it follows from \eqref{eqn:canonical} that its expression is given by
\begin{equation*} 
g(\nabla^B_XY,Z)=g(\nabla^{LC}_XY,Z)-\frac{1}{2}d^c\omega(X,Y,Z),
\end{equation*}
with its torsion $3$-form $H$ 
\[
H(X, Y,  Z) = g(T^B (X, Y), Z) = d \omega (JX, JY, JZ) = - d^c \omega (X, Y, Z), \quad  X, Y, Z \in \Gamma(\operatorname{TM}).
\]
When $dH=0$, then the Hermitian structure $(J,g)$ is said \emph{strong K\"ahler with torsion} (SKT, in short) or \emph{pluriclosed}.
\begin{definition}
A Hermitian manifold $(M,J,g)$ is said \emph{Calabi-Yau with torsion} (CYT, in short) if the restricted holonomy group of the Bismut connection is in $\operatorname{SU}(n)$ or, alternatively, if the Bismut Ricci curvature $\rho^B$ (see \eqref{eqn:rhob}) vanishes identically.
\end{definition}
We introduce now some notations: given a Hermitian structure $(J,g)$ the associated Lee form is expressed as  $\theta=J \delta\omega$. Equivalently, $\theta$ is characterized by the relation $d\omega^{n-1}=\theta \wedge \omega^{n-1}$. Furthermore, $\theta(X)= \frac{1}{2}\sum H(e_i,Je_i, JX)$ (see, for instance, \cite[Lemma 8.6]{GFS}). We will denote by $\theta^\sharp$ the vector field dual to $\theta$ via $g$, i.e., $\theta(Y)=g(\theta^\sharp, Y)$, for any $Y \in \Gamma(\operatorname{TM})$. \par
In the complex case, many traces of the Bismut curvature tensor are possible. The Bismut Ricci tensor is defined as
\[ Ric^B(X,Y)=\sum_i R^B (e_i, X,Y,e_i),\] 
and on a SKT manifold, it is related with $\rho^B$ by 
\begin{equation} \label{eqn:rhob2}
\rho^B(X,Y)=-Ric^B(X,JY)-(\nabla^B_X \theta) JY
\end{equation}
(see \cite{IP} or \cite[Lemma 8.9]{GFS}). When $\nabla^B \theta=0$, the CYT condition is equivalent to have $Ric^B=0$. \par

We recall the following definition 
\begin{definition}
A SKT manifold $(M,J,g)$ is called
\begin{enumerate}
\item \emph{Bismut Hermitian Einstein} (BHE, for short), if it is CYT, 
\item \emph{steady pluriclosed soliton} if there exists a smooth function $f$ such that 
\[
Ric^{LC}-\frac{1}{4} H^2+ \nabla^2 f =0, \quad  \delta H +\iota_{\operatorname{grad} f} H =0,
\]
where $H^2$ is defined as $H^2(X,Y)=g(\iota_X H, \iota_Y H)$ and $\delta=d^*$ is the codifferential of $g$.
\item \emph{generalized Einstein} if $Ric^B=0$, namely, $Ric^{LC}=\frac{1}{4} H^2, \  \delta H =0$.
\end{enumerate}
\end{definition}

BHE metrics arise as stationary points of the pluriclosed flow \cite{ST}, which has motivated significant interest in the construction of compact examples.

Since any BHE manifold  is in particular CYT, the restricted holonomy of the Bismut connection is contained in $\operatorname{SU}(n)$. Another open question regards the existence of (non-K\"ahler) examples with restricted Bismut holonomy exactly $\operatorname{SU}(n)$, where $n$ is the complex dimension of the manifold. In this section we answer negatively to this question.\par
In what follows, we will need the following Proposition 
\begin{proposition} \cite{GFJS, SU, Le2, ALS} \label{prop:BHE}
Any compact BHE manifold $(M,J,g)$ is a steady pluriclosed soliton with a unique normalized potential $f$. The vector fields $V:=\frac{1}{2} (\theta^\sharp- \operatorname{grad} f)$ and $JV$ are holomorphic Killing and Bismut parallel and they vanish if and only if $(J,g)$ is K\"ahler. Moreover, $dV^\flat=\frac{1}{2}d\theta$ and $dJV^\flat=\frac{1}{2}(dJ\theta-dJdf)$ are of type $(1,1)$, and, furthermore, they satisfy $\iota_V dV^\flat=\iota_{JV}  dV^\flat=0$ and $\iota_V dJV^\flat=\iota_{JV}  dJV^\flat=0$. Finally, $\iota_V df=\iota_{JV}  df=0$.
\end{proposition}
The proof of the previous Proposition is described in details in \cite{ALS} and relays on \cite[Proposition 2.16 and Theorem 2.18]{Le2} and \cite[Proposition 4.1]{SU}. \par
BHE structures naturally arise in strong HKT geometry. We start recalling the following
\begin{definition}
An hyper-Hermitian manifold $(M^{4n},I,J,K,g)$  is said \emph{hyper-K\"ahler with torsion} (HKT, in short) if $\nabla^B_{I}=\nabla^B_{J}=\nabla^B_{K}$.
\end{definition}
Analogously, the HKT condition can be rephrased by saying that the Bismut torsions $H_L=-d^c_L \omega_L, \ L=I,J,K,$ coincides. To simplify the notation, we set $\nabla^B:=\nabla^B_{I}=\nabla^B_{J}=\nabla^B_{K}$ and $H:=H_I=H_J=H_K$. \\
Since $\nabla^B I=\nabla^B J=\nabla^B K=0,$ and $\nabla^B g=0$, it is straightforward to note that $\operatorname{Hol}(\nabla^B) \subseteq \operatorname{Sp}(n)$ \cite{HP}, where $n$ is the quaternionic dimension of the manifold. So, exploiting that $\operatorname{Sp}(n)=\operatorname{Sp}(2n,\C) \cap \operatorname{SU}(2n)$ we get that the three Hermitian structures $(I,g),(J,g),(K,g) $ are CYT. \par
We recall that given a HKT manifold, there exists a unique torsion free connection $\nabla^{Ob}$ on $\Gamma(\operatorname{TM})$ such that $\nabla^{Ob} I=\nabla^{Ob} J=\nabla^{Ob} K=0$. This connection is called the \emph{Obata connection.} The expression of the Obata in terms of the Bismut can be found in \cite{IP2}. Let $A$ be the tensor  defined as
\begin{equation} \label{eqn:A}
2A(X,Y,Z):=-H(X,IY,IZ) - H(IX,IY,Z)-H(X,KY,KZ)-H(IX,KY,JZ).
\end{equation}
Then, by \cite[Proposition 3.1]{IP2},
\begin{equation}\label{eqn:obata}
g(\nabla^{Ob}_X Y,Z)=g(\nabla^B_X Y,Z)+ A(X,Y,Z).
\end{equation}
\begin{definition}
An HKT manifold $(M,I,J,K,g)$ is said \emph{strong hyper-K\"ahler with torsion} (strong HKT, in short) if $dH=0$. 
\end{definition}
In particular, if $(I,J,K,g)$ a strong HKT structure, then each Hermitian structure $(I,g),(J,g),(K,g)$ is BHE. \par
Given a hyper-Hermitian manifold $(M,I,J,K,g)$, the  Lee forms $\theta_I, \theta_J,\theta_K$ coincide \cite{FG} (this was previously known for the HKT case by \cite{IP}), and hence, also the functions $\delta \theta_I, \delta \theta_J,\delta \theta_K$. Also in this case, we will set $\theta:=\theta_I=\theta_J=\theta_K$. \par
In what follows, we will need the following Proposition
\begin{proposition} \label{prop:V}
Let $(M,I,J,K,g)$ be a compact strong HKT manifold, and let $V_I=\frac{1}{2}(\theta^\sharp-\operatorname{grad} f_I),V_J=\frac{1}{2}(\theta^\sharp-\operatorname{grad} f_J),V_K=\frac{1}{2}(\theta^\sharp-\operatorname{grad} f_K)$ be the vector fields respectively associated to $(I,g),(J,g),(K,g)$ defined in Proposition \ref{prop:BHE}. Then $V_I=V_J=V_K:=V$. 
\end{proposition}
\begin{proof}
Since $V_L=\frac{1}{2}(\theta^\sharp-\operatorname{grad} f_L)$ is Killing (see Proposition \ref{prop:BHE}), $\cal{L}_{V_{L}} g=0$, for any $L=I,J,K$. Fixed any local orthonormal basis $\{e_1,\dots,e_{4n}\}$, if we take the trace \[\sum_i\cal{L}_{V_{L}} g(e_i,e_i)=0\] we get 
\begin{equation}\label{en:laplacian}
\delta \theta+ \Delta f_L=0,
\end{equation}
for each $L=I,J,K$. Hence, $\Delta f_I=\Delta f_J=\Delta f_K$, and so by the maximum principle the functions $f_L$ must differ from each other by a constant. 
\end{proof}
Without loss of generality, we may assume that $f_I=f_J=f_K:=f$. \par
In \cite{ALS} the authors pointed out the lack of compact (non-K\"ahler) examples of BHE manifolds with $\operatorname{Hol}^0(\nabla^B)=\operatorname{SU}(n)$. Using the results in \cite{GFJS} (see also \cite{ALS})   it is possible to show that on a compact BHE manifoldsthe holonomy of $\nabla^B$ cannot be full, unless the structure is K\"ahler. This result is implicit in \cite{GFJS} when combined with Proposition 4.24 of \cite{GFS}. It has been also pointed out in \cite{Pap2}. With the aim to extend the result for the strong HKT manifolds, we include the detailed proof for the sake of completeness:
\begin{theorem} \label{thm:holonomy}
\begin{enumerate}
\item Let $(M,J,g)$ be a compact non-K\"ahler BHE manifold of complex dimension $n$. Then $\operatorname{Hol}^0(\nabla^B) \subseteq \operatorname{SU}(n-1)$. In particular, if a compact BHE manifold has restricted holonomy exactly $\operatorname{SU}(n)$, then it is K\"ahler.
\item Let $(M,I,J,K,g)$ be a compact strong HKT manifold which is not hyper-K\"ahler of quaternionic dimension $n$. Then $\operatorname{Hol}(\nabla^B)\subseteq \operatorname{Sp}(n-1)$. In particular, if a strong HKT manifold has holonomy exactly $\operatorname{Sp}(n)$, then it is hyper-K\"ahler.
\end{enumerate}
\end{theorem}
\begin{proof}
We prove the first assertion in details. Given a compact non-K\"ahler BHE manifold $(M,J,g)$, the vectors $V$ and $JV$ satisfy $\nabla^B V=0$ and $\nabla^B J V=0$ and, furthermore, they have constant and non-zero norm (see Proposition \ref{prop:BHE}). \par
Let us fix any point $p \in M$ and any null-homotopic loop $\gamma$ centered at $p$.  The vectors $Z:=V_p$ and $JZ=(JV)_p$ are always non-zero and, exploiting that $\nabla^B V=0$ and $\nabla^B J V=0$, they satisfy $P_\gamma Z=Z$ and $P_\gamma (JZ)=JZ$. Therefore, if $(J,g)$ is non K\"ahler, at any point $p$ there exists a complex vector $W$ which is non zero and such that $P_\gamma W=W$, for any null-homotopic loop $\gamma$ centered at $p$. This implies that the holonomy of the Bismut connection of any non-K\"ahler BHE manifold reduces from $\operatorname{SU}(n)$ to $\operatorname{SU}(n-1)$, where, to the sake of clarity, we specify that we consider $\operatorname{SU}(n-1)$ as a subgroup of $\operatorname{SU}(n)$ in the following way:
$$
\operatorname{SU}(n-1) \to \operatorname{SU}(n),  \quad 
A  \mapsto \begin{pmatrix}
                           1 &   0 \\
                           0 &   A \\
                           \end{pmatrix}.
$$ 
In particular, if there exists a compact BHE manifold with $\operatorname{Hol}^0(\nabla^B) = SU(n)$, it has to be necessarily K\"ahler. \par
The proof of the second statement is analogous. In fact, exploiting that  $\nabla^BV=0$ and  $\nabla^B IV=\nabla^B JV=\nabla^B KV=0$, we get that $\norm{V}^2, \norm{IV}^2, \norm{JV}^2, \norm{KV}^2$ are constant and non zero, as the manifold is not hyper-K\"ahler. We fix any point $p \in M$ and we set $Z:=V_p, \ IZ=(IV)_p\, JZ=(JV)_p, \ KZ=(KV)_p \in \operatorname{T_pM}\setminus\{0\}$. Then, for any loop $\gamma$ based at $p$, $P_\gamma Z=Z, \ P_\gamma I Z=IZ, P_\gamma JZ=JZ, \ P_\gamma KZ=KZ$. From here the proof is straightforward.
 \end{proof}
 
 \begin{remark} Note that non-compact strong HKT examples which are not Bismut-flat are known \cite{GPS}. Recently a non-compact example in dimension 6 with a complete BHE metric has been constructed \cite{G-F}, followed by a compact one in \cite{ALL-26}. 
 \end{remark}

In the last part of this section we rule out the existence of (non-K\"ahler) BHE metrics on some classes of manifolds as, for instance, solvmanifolds, i.e. on compact quotients of a $1$-connected solvable Lie group $G$ by a discrete co-compact subgroup $\Gamma$. 
\begin{corollary} \label{corollary:solvmanifolds}
Let $\Gamma \backslash G$ be a solvmanifold endowed with an invariant Hermitian structure $(J,g)$.  Then the following are equivalent:
\begin{enumerate}[label=\arabic*.]
    \item $(J,g)$ is BHE,
    \item $(J,g)$ is K\"ahler.
\end{enumerate}
In particular, given a solvmanifold $\Gamma \backslash G $ endowed with a invariant  hyper-Hermitian structure $(I, J, K,g)$, then $(I,J,K,g)$ is strong HKT if and only if it is hyper-K\"ahler.
\end{corollary}
\begin{alignedproof}{2.\gives1.}
\alignedproofstep{2\gives1}
It follows by the characterization of K\"ahler solvmanifolds given in \cite{Ha}, which are flat.
\alignedproofstep{1\gives2}
Assume that $(J,g)$ is BHE. Then, taking the trace of $\cal{L}_V g=0$ as before, we get that $\Delta f + \delta \theta=0$. Since $(g,J)$ is left-invariant, then it is Gauduchon, and so $f$ must be constant by the maximum principle, implying that $$\nabla^B\theta^\sharp=2\nabla^BV + \nabla^B \operatorname{grad} f=2 \nabla^BV=0.$$
By \eqref{eqn:rhob2}, on a compact Hermitian SKT manifold 
\[
\rho^B(X,Y)=-Ric^B(X,JY)-(\nabla^B_X \theta) JY.
\]
Exploiting that $\rho^B=0$ and $\nabla^B \theta=0$, we get that $Ric^B=0$.\par
Moreover, if $Ric^B=0$ then the symmetric part of  $Ric^B$ vanishes, i.e., $Ric^{LC}=\frac{1}{4} H^2$. Therefore
 \[
scal^{LC}=\frac{1}{4} \norm{H}^2 \ge 0.
  \]
Since for any left-invariant metric on a solvable Lie group $scal^{LC} \le 0$ \cite{Je, Mil}, this forces $scal^{LC}=0$. Hence $H=0$, concluding the proof. 
\qedhere
\end{alignedproof}
\begin{remark}
We observe that with the same proof one can prove that any compact BHE manifold $(M,J,g)$ such that $(J,g)$ is Gauduchon and $g$ has non positive scalar curvature must be K\"ahler. 
\end{remark}
 \section{Strong HKT manifolds with parallel Bismut torsion} \label{section3}
 Up to now, very little is known about the geometry of strong HKT manifolds. In the compact case the only examples are due to Joyce \cite{Jo},  Barberis-Fino \cite{BF}, Brienza-Fino-Grantcharov \cite{BFG2}, and recently Witten \cite{W}. The examples of Joyce \cite{Jo} and Barberis-Fino \cite{BF} are homogeneous and Bismut flat, while the examples constructed in \cite{BFG2} are local products of a Bismut flat manifold with a hyper-K\"ahler one.  The structure of the examples in \cite{W} is not understood yet. In the non-compact case a non-homogeneous example has been found by Moraru and Verbitsky \cite{MV}. \par
 We point out that in each compact example the Bismut torsion $H$ is parallel with respect to the Bismut connection: in fact, in each of the cases above the Bismut connection satisfies the first Bianchi identity. In this section we characterize all the compact strong HKT manifolds with $\nabla^B H=0$. \par
 We recall the following definition.
 \begin{definition}
 A Hermitian manifold $(G'=G \times \R^k, J_L,b')$ is a \emph{Samelson space} if $G$ is a $1$-connected compact semisimple Lie group, $b'$ is a bi-invariant metric on $G'$ and $J_L$ is a left invariant complex structure such that $b'(J_L \cdot, J_L \cdot)=b'(\cdot,\cdot)$. 
 \end{definition}
We recall that on any $G'$ endowed with a bi-invariant metric $b'$ there exists a left-invariant complex structure compatible with $b'$.   In fact, $G'$ shares the same Lie algebra with $G''=G \times \mathbb{T}^k$, namely, $\operatorname{Lie}(G')=\operatorname{Lie}(G'')$. Since $b'$ is bi-invariant, we may assume that $G''$ is endowed with the same metric $b'$ and so, there exists a left-invariant complex structure $J_L$ which satisfy $b'(J_L \cdot, J_L \cdot)=b'(\cdot,\cdot)$ \cite{Sa}. Since $J_L$ is left-invariant, it is defined on $G'$ as well.  
\begin{remark}
Consider the left-invariant connection $ \nabla $ on the Samelson space $G' $ such that all left-invariant vector fields are parallel. Since the metric $ g $ and the complex structure $ J_L$ are left-invariant they are preserved by $ \nabla $. The torsion tensor is $ T^\nabla(X,Y)=[X,Y] $, therefore
\begin{equation}\label{HKT:eq:Joyce_strong}
H(X,Y,Z)=b'([X,Y],Z)
\end{equation}
which is known to be a $ 3 $-form. It follows that the Bismut connection of  $(J_L, b')$ does not depend on the complex structure, as long as the complex structure is invariant. In particular, $dH=0$, since $H$ is a bi-invariant $3$-form. 
\end{remark}
 We may enlarge the previous definition to hyper-Hermitian structures in the following way
  \begin{definition}
 A hyper-Hermitian manifold $(G'=G \times \R^k, J_{L_{I}},J_{L_{J}},J_{L_{K}},b')$ is a \emph{hyper-Hermitian Samelson space} if $G$ is a $1$-connected compact semisimple Lie group, $b'$ is a bi-invariant metric on $G'$ and $J_{L_{I}},J_{L_{J}},J_{L_{K}} $ are left invariant complex structures compatible with $b'$. 
 \end{definition}

We point out that the hypercomplex structure $(J_{L_I}, J_{L_J}, J_{L_K})$ is a Joyce hypercomplex structure in the sense of \cite{Jo}. Indeed, $(J_{L_I}, J_{L_J}, J_{L_K})$ defines a left-invariant hypercomplex structure on $G'$, and hence on $G'' = G \times \mathbb{T}^k$. Since $G''$ is a compact Lie group and any left-invariant hypercomplex structure on a compact Lie group is of Joyce type, the claim follows.
 \par
Furthermore, by the previous remark we easily get that any hyper-Hermitian Samelson space is strong HKT. \par
We recall that for a compact BHE manifold with parallel Bismut torsion, the following splitting result holds:
 \begin{theorem}[\cite{BFG,BPT}] \label{thm:splitting}
Let $(M,J,g)$ be a compact BHE manifold satisfying $\nabla^B H=0$. Then its Riemannian holomorphic universal cover splits isometrically and holomorphically as a product of a K\"ahler Ricci flat manifold and a Samelson space. 
\end{theorem}
\begin{remark}
We  point out that any compact BHE manifold with parallel Bismut torsion has non-negative Ricci. This is a consequence of the fact that if $H$ is parallel, then so is the Lee form $\theta$, and therefore $Ric^B=0$ (see Equation \eqref{eqn:rhob2}). Since $Ric^B=0$ then also its symmetric part must vanish, and therefore $Ric^{LC}=\frac{1}{4} H^2 \ge 0$. 
\end{remark}
\begin{remark}
Observe that we may always assume that the K\"ahler factor is compact. In fact, the universal cover of a compact K\"ahler Ricci flat manifold splits uniquely as a product of the kind $\C^k$ and $N$, where $N$ is the product of manifolds with holonomy $\operatorname{Sp}(m)$ and $\operatorname{SU}(r)$ \cite{Be}. Now, since by the assumption of the Theorem  \ref{thm:splitting} $(M,J,g)$ is compact, BHE and with parallel Bismut torsion, it has non-negative Ricci and so its universal cover splits as a product of a compact manifold and a vectorial factor \cite{CG}. Hence, $N$ is compact and the K\"ahler vectorial factor can be absorbed by the Bismut flat part. 
\end{remark}
By Theorem \ref{thm:splitting}, together with the fact that any compact BHE manifold with parallel Bismut torsion has non-negative Ricci, we get the following result.
 \begin{corollary}
 Any compact BHE manifold with parallel Bismut torsion is formal according to Sullivan.
 \end{corollary}
 \begin{proof}
 Let $(M,g)$ be a compact manifold with non-negative Ricci and let $\widetilde{M}$ be its universal cover. By \cite{Mi}, if $\widetilde M$ is formal, then $M$ is formal. Since any BHE metric $g$ with parallel Bismut torsion has non-negative Ricci, it suffices to look at the formality of the universal cover. \par
 By Theorem \ref{thm:splitting}, the universal cover of a compact BHE manifold with parallel Bismut torsion is diffeomorphic to a product of a compact K\"ahler manifold and a Samelson space which are both formal. Then the result follows by the fact that a product manifold is formal if and only if each factor is formal.
 \end{proof}
 Now, we extend Theorem \ref{thm:splitting} to compact strong HKT manifolds. Just for the proof of the following two Theorems we will denote the hyper-complex structure by $J_1, J_2, J_3$ instead of $I,J,K$ since the notation with the subscript suits better for the proofs.  
 \begin{theorem} \label{thm:splitting2} 
 Let $(M, J_1, J_2,J _3, g)$ be a compact strong HKT manifold satisfying $\nabla^B H=0$ and let $(\widetilde M, \widetilde J_a, \widetilde g)$ its universal cover. Then, $(\widetilde M,\widetilde J_1, \widetilde J_2, \widetilde J_3,\widetilde g)$ splits as a product of a compact hyper-K\"ahler manifold and a hyper-Hermitian Samelson space.
 \end{theorem}
 \begin{proof}
If we apply Theorem \ref{thm:splitting} to each Hermitian structure $(J_a, g)$, we get that for any $a=1, 2,3$,
\[
(\widetilde M, \widetilde J_a, \widetilde g)=(K_a,I_a,k_a) \times (G'_a,I_{L_{a}},b'_a),
\]
 where $(K_a,I_a,k_a)$ is the compact K\"ahler Ricci flat factor and $(G'_a,I_{L_{a}},b'_a)$ is a Samelson space. We recall that $G'_a = G_a \times \R^{s_{a}}$, where $G_a$ is a $1$-connected compact semisimple Lie group and that the bi-invariant metric $b'_a$ is the product of a bi-invariant metric $b_a$ on $G_a$ and a flat metric $h_a$ on $\R^{s_{a}}$ (see \cite{Mil}). \par
 We first observe that $s:=s_1=s_2=s_3$ and $h:=h_1=h_2=h_3$. This is a straightforward consequence of the uniqueness of de De Rham splitting and the Cheeger-Gromoll Theorem \cite{CG}. Analogously, $(K_1,k_1)=(K_2,k_2)=(K_3,k_3)$ and $(G_1,b_1)=(G_2,b_2)=(G_3,b_3)$. We explain the first in detail, the second is analogous. \par
We prove that $(K_1,k_1)=(K_2,k_2)$, the proof proceeds in the same way for $(K_2,k_2)=(K_3,k_3)$. Since any irreducible de Rham factor of $(K_1,k_1)$ is still K\"ahler, by the uniqueness of de De Rham splitting, any irreducible de Rham factor of $(K_1,k_1)$ is identified with a unique irreducible de Rham factor of $(K_2,k_2)$. This forces $(K,k):=(K_1,k_1)=(K_2,k_2)=(K_3,k_3)$, and so $(K,k)$ is hyper-K\"ahler as it is preserved by $I_1, I_2, I_3$, which satisfy $I_1I_2=-I_2I_1=I_3$. \par
With the same argument, one can prove that $(G,b):=(G_1,b_1)=(G_2,b_2)=(G_3,b_3)$, and so $(G'=G \times \R^s, b'=b+h, I_{L_{a}})$ is a hyper-Hermitian Samelson space.
 \end{proof}
 As an application of this result we may characterize a finite cover of a strong HKT manifold with parallel Bismut torsion.
 \begin{theorem}
Any compact strong HKT manifold with parallel Bismut torsion $(M, J_1, J_2, J_3, g)$  is, up to a finite cover, the product of a compact hyper-K\"ahler manifold and a compact Bismut flat one. 
 \end{theorem}
 \begin{proof}
 Let $(\widetilde M, \widetilde J_a, \widetilde g)=(K, I_a,k) \times (G',I_{L_{a}},b')$ be the universal cover of the strong HKT manifold with parallel Bismut torsion.  By the previous Theorem, $\widetilde M \cong K \times G \times \R^s$, where $K$ is a compact hyper-K\"ahler manifold, $G$ is a compact semisimple Lie group and $\R^s$ is endowed with its flat metric. Since $K$ is hyper-K\"ahler, any de Rham factor of $K$ is hyper-K\"ahler, implying that no irreducible de Rham factor of $K$ can be isometric to an irreducible de Rham factor of $G$. Furthermore, since $K$ is compact, no irreducible de Rham factor of $K$ can be isometric to some $\R^{4r}$. From this we get that 
 \[
 \text{Iso}(\widetilde M, \widetilde g) \cong  \text{Iso}(K) \times  \text{Iso}(G) \times  \text{Iso}(\R^s).  
 \]
 The fundamental group $\pi_1(M)$ acts on $(\widetilde M, \widetilde J_a, \widetilde g)$ via hyper-holomorphic isometries. Let $\sigma$ be such an automorphism; by the discussion above, there exist three isometries $\sigma_1, \sigma_2$ and $\sigma_3$ such that 
 \[
 \sigma (p,g,t)=(\sigma_1 (p), \sigma_2 (g), \sigma_3(t)), \ \text{for} \ p \in K,\ g \in G \ \text{and} \ t \in \R^s,
 \]
 and we point out that $\sigma_1$ is a hyper-K\"ahler isometry of $(K, I_a,k)$. \par
 We consider the group homomorphism $\tau: \sigma \to \sigma_1$ from $\pi_1(M)$ to the group of hyper-holomorphic isometries of $(K, I_a,k)$. Let $\Gamma$ to be its kernel. Then $\Gamma$ is normal in $\pi_1(M)$, as it is the kernel of a group homomorphism and, furthermore, it is finite. In fact $\frac{\pi_1(M)}{\Gamma} \cong \operatorname{Im}(\tau)$ which is finite, as so is the group of hyper-holomorphic isometries of $(K, I_a,k)$. \par
 Observe that if $\Gamma$ is trivial, then $\pi_1(M)$ is itself finite, and, hence, the covering $\widetilde M \to M$. \par
 Since $\Gamma$ is a finite normal subgroup of $\pi_1(M)$, we may consider the finite (and hence compact) cover $M'=\widetilde M/ \Gamma \to M$, which splits as a product of $(K, I_a,k) $ and a Bismut flat space.
  \end{proof}
  \begin{remark}
An example of a compact strong HKT manifold with parallel Bismut torsion, which is not globally a product of a Bismut-flat manifold and a hyper-K\"ahler one, was exhibited in \cite{BFG2}. It is given by a mapping torus of the product of the Fermat quartic, endowed with its hyper-K\"ahler metric, and $\mathrm{SU}(2)$. Moreover, Remark~4.1 in \cite{BFG2} shows that this manifold admits a finite cover which splits as a product of the Hopf surface and the Fermat quartic.
 \end{remark}
 As a Corollary, we immediately get that
 \begin{corollary}
Any compact strong HKT manifold with parallel Bismut torsion is formal according to Sullivan.
 \end{corollary}
\begin{remark}
We point out that an analogous result does not hold for compact balanced HKT manifolds. Indeed, it has been shown by Dotti and Fino \cite{DF} that there exist a (non hyper-K\"ahler) $8$-dimensional balanced HKT nilmanifold with parallel Bismut torsion.
\end{remark}
\section{Ricci foliation of an Hypercomplex manifold} \label{section4}
 Let $(M^{4n},I,J,K)$ be a compact hypercomplex manifold and let $\nabla^{Ob}$ be the Obata connection, i.e., the unique torsion free connection satisfying $\nabla^{Ob} I=\nabla^{Ob} J=\nabla^{Ob} K=0$. We recall the following
 \begin{definition}
Let $(M^{4n},I,J,K)$ be a hypercomplex manifold. A $(2n,0)$ form $\alpha$ is said to be $q$-real if $J \overline{\alpha}=\alpha$.
 \end{definition}
 Let $F(\operatorname{TM})$ be the frame bundle of $M$ and consider the character $det:\operatorname{GL}(n,\mathbb{H}) \to \R_{>0}$ induced by the Dieudonnè determinant \cite{Di}. The associated bundle $K_\R$ is the real line bundle of $(2n,0)$ $q$-real forms on $M$. More precisely, fixing any hyper-Hermitian metric $g$ on $M$, $K_\R$ is the real line bundle generated by the $(2n,0)$ form $\Omega^n$, where $\Omega=\omega_J + i \omega_K$, where $\omega_J=g(J \cdot, \cdot)$ and $\omega_K=g(K\cdot, \cdot)$. Clearly, the canonical bundle $K_{(M,I)}$ is obtained as a complexification of $K_\R$. 
 \begin{definition}
 The \emph{Obata Ricci curvature} $\Theta$ is the curvature of the Obata connection on $K_\R$.
 \end{definition}
 \begin{proposition}
 The Obata Ricci curvature $\Theta$  of a compact hypercomplex manifold $(M,I,J,K)$  is $\operatorname{SU}(2)$-invariant and exact.
 \end{proposition}
 \begin{proof}
Since $K_{(M,I)}=K_\R \otimes \C$, $\Theta$ is the curvature of the Obata connection on the canonical bundle of $(M,I)$. Fixing any hyper-Hermitian metric $g$, $K_{(M,I)}$ admits a non-vanishing section $\Omega^n$ as above, and $\nabla^{Ob}_{K{(M,I)}} \Omega^n= \alpha  \otimes \Omega^n$,  where $\alpha=\gamma + \overline{\gamma}$, $\gamma \in \Omega^{1,0} M$ and $K_{(M,I)}$ is the canonical bundle obtained as a complexification of $K_\R$. Then, 
 \begin{equation} \label{eqn:c1}
 \Theta= d \alpha + \alpha \wedge \alpha =d \alpha
 \end{equation} 
 coincides with the Ricci of the Obata connection on $\operatorname{TM}$, and, therefore, it is $(1,1)$ with respect to any complex structure $L = I, J, K$ , i.e., it is $\operatorname{SU}(2)$-invariant. Since $\Omega$ is a non-degenerate $(2,0)$ form, $c_1(M,I)=0$ and  $\Theta$ is exact. 
\end{proof} 
We recall the following Lemma
\begin{lemma}
Let $(V,I,J,K)$ be a real vector space equipped with a quaternionic action and let $\eta \in \bigwedge^{1,1}_I V$ be a real $\operatorname{SU}(2)$-invariant $(1,1)$ form. Then the eigenvalues of $\eta$ occur in pairs $\sigma_i, \ -\sigma_i$, where the eigenvalues are meant in the sense of eigenvalues of a skew-symmetric endomorphism. 
\end{lemma}
\begin{proof}
Let us take the unitary basis $\{\psi_1=\varphi_1 + i I \varphi_1,\dots,\psi_{2n}=\varphi_{2n}+i I \varphi_{2n}\}$ of $V^{{1,0}^{*}}$ which diagonalizes $\eta$, i.e., such that \
\[
\eta= i \sum_{j=1}^{2n} \lambda_j \psi^{j} \wedge \psi^{\overline{j}}=2  \sum_{j=1}^{2n} \lambda_j \varphi^{j} \wedge I \varphi^{j}=J \eta= -2  \sum_{j=1}^{2n} \lambda_j J\varphi^{j} \wedge I(J \varphi^{j}).
\]
Since $J$ takes the positive form $ \varphi^{j} \wedge I \varphi^{j}$ in the negative form $-J\varphi^{j} \wedge I(J \varphi^{j})$ and since $\eta$ and $J\eta=\eta$ have the same eigenvalues, they must occur with opposite signs.
\end{proof}
Let $r$ be the rank of $\Theta$ at a general point of $M$. Since by previous Lemma the eigenvalues of $\Theta$ go in pairs, then $r < 2n$. In fact, if $r=2n$, then $\Theta^{2n}$ would be a positive $(2n,2n)$ form if $n$ is even and a negative $(2n,2n)$ form otherwise, implying that
\[
\int_{M} \Theta^{2n} > 0 \ \text{in the first case, \ and} \ \int_{M} \Theta^{2n} < 0 \ \text{in the second one},
\]
which is impossible, as $\Theta$ is exact.
\begin{definition}
The sub-sheaf $\ker(\Theta) \subset \operatorname{TM}$  is involutive, as $d\Theta=0$. We call $\ker(\Theta)$ the \emph{Ricci Foliation} of an hypercomplex manifold $(M,I,J,K)$.
\end{definition}
 \begin{remark}
 The leaves of the Ricci foliation are always hypercomplex manifolds with restricted Obata holonomy contained in $\operatorname{SL}(n,\mathbb{H})$. 
 \end{remark}
\section{Ricci foliation of strong HKT manifolds} \label{section5}
Let us consider an HKT manifold $(I,J,K,g)$. It has been shown in \cite{IP2} that the Ricci tensor of the Obata connection coincides with the differential of the Lee form, and hence, the Obata Ricci curvature of $K_\R$ satisfies $\Theta=d\theta$. In this case the Ricci foliation  is hence given by the kernel of $d\theta$. 
\begin{proposition}	\label{prop:dtheta} 
Let $(M,I,J,K,g)$ be a compact simply connected strong HKT manifold. If there exists a point $p$ and an open neighborhood $\cal{U}$ of $p$ on which $d\theta=0$, then the manifold is hyper-K\"ahler. 
\end{proposition}
\begin{proof}
Locally, any strong HKT metric $\Omega$ has a potential $\mu$ which satisfies $dd^Id^Jd^K\mu=0$. In particular, $\mu$ is real analytic and so $\Omega$ is real analytic (see \cite[Lemma 9.4]{MV2}), where $\Omega$ is the $(2,0)$ form defining the strong HKT structure. Therefore, $\theta$ is real analytic and so $d\theta$ must vanish everywhere. Since $M$ is simply connected, the result follows by \cite[Corollary 5.3]{IP}.
\end{proof}
We recall the following well known 
\begin{lemma} \label{lem:useful}
Let $(M,J,g)$ be an Hermitian manifold. Then, for any $X,Y,Z \in  \Gamma(TM)$,
\[
H(X,Y,Z)=H(JX,JY,Z)+H(JX,Y,JZ)+H(X,JY,JZ).
\]
\end{lemma}
A proof can be found in \cite[Lemma 2.1]{BPT}. \par 

\begin{remark} \label{rmk:useful}
If $(M,I,J,K,g)$ is a compact (non-hyper-K\"ahler) strong HKT we have already noticed that $(I,g)$, $(J,g)$, $(K,g)$ are BHE structures with the same soliton potential $f$, namely, we the vector fields $V_I,V_J,V_K$ defined in Proposition \ref{prop:BHE} coincide (see Proposition \ref{prop:V}). In particular, we get for free the following useful properties:
\begin{enumerate}[label=\arabic*.]
\item[1)] $V,IV,JV,KV$ are non zero, Killing and Bismut parallel, 
\item[2)]  $\cal{L}_V I=\cal{L}_V J=\cal{L}_V K=0$,
\item[3)] $\cal{L}_{IV} I=\cal{L}_{JV} J=\cal{L}_{KV} K=0$,
\item[4)] $\iota_VdV^\flat=\iota_{IV}dV^\flat=\iota_{JV}dV^\flat=\iota_{KV}dV^\flat=0$. Furthermore, $dV^\flat$ is $\operatorname{SU}(2)$-invariant 
\item[5)] $\iota_VdIV^\flat=\iota_{IV}dIV^\flat =0$  and $dIV^\flat$ is $(1,1)$ with respect to $I$, 
\item[6)]  $\iota_VdJV^\flat=\iota_{JV}dJV^\flat=0$ and $dJV^\flat$ is $(1,1)$ with respect to $J$,
\item[7)] $\iota_VdKV^\flat=\iota_{KV}dKV^\flat=0$ and $dKV^\flat$  is $(1,1)$ with respect to $K$,
\item[8)]  $\iota_Vdf=\iota_{IV}df=\iota_{JV}df=\iota_{KV}df=0$.
\end{enumerate}
\end{remark}

From   properties  1),  2) and 3),  we can prove the following

\begin{theorem}\label{prop:new}
Let $(M,I,J,K,g)$ be a compact strong HKT manifold which is not hyper-K\"ahler.  If  the distribution spanned by $\{V, IV, JV, KV\}$ is involutive,  then its Lie algebra is isomorphic to $\mathfrak{u}(1)\oplus\mathfrak{su}(2)$ or $\mathbb{R}^4$.
\end{theorem}

\begin{proof}
Since the vector fields $V, IV, JV, KV$ are Bismut-parallel, they have constant, non-zero norm; up to rescaling, we may assume that each of them has unit length. We thus obtain an orthogonal decomposition
\[
TM = \mathcal{F} \oplus^\perp \mathcal{F}^\perp,
\]
where $\mathcal{F} = \mathrm{span}\{V, IV, JV, KV\}$ and $\mathcal{F}^\perp$ denotes the horizontal distribution.

By statement~3) of Remark~\ref{rmk:useful}, we have
\[
\mathcal{L}_V I = \mathcal{L}_V J = \mathcal{L}_V K = 0,
\]
which implies
\[
[V, IV] = [V, JV] = [V, KV] = 0.
\]
Together with the identities
\[
\nabla^B V = \nabla^B IV = \nabla^B JV = \nabla^B KV = 0,
\]
this yields
\begin{align} \label{eqn:computation}
g([IV, JV], V) 
&= g\bigl(\nabla^B_{IV} JV - \nabla^B_{JV} IV - T^B(IV,JV), V\bigr)
 = - H(IV, JV, V) = 0, \\
g([IV, JV], IV) &= - H(IV, JV, IV) = 0, \nonumber \\
g([IV, JV], JV) &= - H(IV, JV, JV) = 0, \nonumber \\
g([IV, JV], KV) &= - H(IV, JV, KV). \nonumber
\end{align}
It follows that
$[IV, JV] = a\, KV,$
where 
\begin{equation} \label{def:a}
 a = - H(IV, JV, KV).
\end{equation}
By the same argument, one obtains
\[
[JV, KV] = a\, IV,
\qquad
[KV, IV] = a\, JV.
\]

Finally, since $V, IV, JV, KV$ are Bismut parallel and, for any $L\in\{I,J,K\}$, the exterior derivative $d(LV^\flat)$ coincides with the skew-symmetric part of $\nabla^{LC} LV^\flat$, we obtain the relations
\begin{equation} \label{eqn:diff}
dV^\flat = \iota_V H, 
\qquad 
d(IV^\flat) = \iota_{IV} H, 
\qquad 
d(JV^\flat) = \iota_{JV} H,
\qquad 
d(KV^\flat) = \iota_{KV} H.
\end{equation}
Using the identity $d(JV^\flat)=\iota_{JV} H$, we show that
\begin{equation} \label{eqn:dJV}
d(JV^\flat) = - a\, KV^\flat \wedge IV^\flat + \beta_J,
\end{equation}
where $a\in \mathcal{C}^\infty(M)$ and $\beta_J$ is a horizontal $2$-form, namely $
\beta_J \in \bigwedge\nolimits^2 \mathcal{F}^\perp .$

Indeed, since $d(JV^\flat)=\iota_{JV} H$, for any vector field $X\in \Gamma(TM)$ we compute
\[
\begin{aligned}
d(JV^\flat)(V,X)
&= H(JV,V,X)
 = - g([JV,V],X) = 0, \\
d(JV^\flat)(IV,X)
&= H(JV,IV,X)
 = - g([JV,IV],X)
 = a\, g(KV,X), \\
d(JV^\flat)(JV,X)
&= H(JV,JV,X) = 0, \\
d(JV^\flat)(KV,X)
&= H(JV,KV,X)
 = - g([JV,KV],X)
 = - a\, g(IV,X),
\end{aligned}
\]
where we used the commutation relations established in \eqref{eqn:computation}.  
These identities imply the decomposition \eqref{eqn:dJV}. \par
Applying the exterior differential to both sides of \eqref{eqn:dJV}, we obtain
\[
0 = d^2(JV^\flat)
= - da \wedge KV^\flat \wedge IV^\flat
 - a\, d(KV^\flat) \wedge IV^\flat
 + a\, KV^\flat \wedge d(IV^\flat)
 + d\beta_J .
\]
Hence,
\begin{equation}\label{eqn:da}
da \wedge KV^\flat \wedge IV^\flat
= - a\, d(KV^\flat) \wedge IV^\flat
  + a\, KV^\flat \wedge d(IV^\flat)
  + d\beta_J .
\end{equation}

By statements~5) and~7) of Remark~\ref{rmk:useful}, we have
\[
\iota_{IV} d(IV^\flat) = 0,
\qquad
\iota_{KV} d(KV^\flat) = 0.
\]
Evaluating \eqref{eqn:da} on the triple $(KV, IV, Z)$, with $Z\in \mathcal{F}^\perp$, we obtain
\[
da(Z) = d\beta_J(KV,IV,Z).
\]
Since $\beta_J$ is horizontal and $[KV,IV]=a\, JV$, we compute
\[
\begin{aligned}
d\beta_J(KV,IV,Z)
&= \mathcal{L}_{KV}\bigl(\beta_J(IV,Z)\bigr)
 - \mathcal{L}_{IV}\bigl(\beta_J(KV,Z)\bigr)
 - \mathcal{L}_{Z}\bigl(\beta_J(KV,IV)\bigr) \\
&\quad - \beta_J([KV,IV],Z)
 + \beta_J([KV,Z],IV)
 + \beta_J([IV,Z],KV) \\
&= 0,
\end{aligned}
\]
and therefore $da(Z)=0$ for all $Z\in \mathcal{F}^\perp$.

It remains to show that $da$ vanishes on $\mathcal{F}$.  
Applying the Jacobi identity to the triple $(V,IV,KV)$ and using statement~$2)$ of Remark~\ref{rmk:useful}, we obtain
\[
0 = [[V,IV],KV] + [[IV,KV],V] + [[KV,V],IV]
   = [a\, JV, V]
   = V(a)\, JV,
\]
which implies $da(V)=0$.
Similarly, applying the Jacobi identity to $(IV,JV,KV)$ yields
\[
da(IV)\, IV + da(JV)\, JV + da(KV)\, KV = 0,
\]
and hence $da(IV)=da(JV)=da(KV)=0$. Consequently, $a$ is constant.

Therefore, the Lie subalgebra of $\Gamma(TM)$ spanned by $\{V, IV, JV, KV\}$ is isomorphic to $\mathbb{R}^4$ if $a=0$, and to $\mathfrak{u}(1)\oplus\mathfrak{su}(2)$ otherwise.

\end{proof}

The following theorem  gives  some structural result about compact simply connected $8$-dimensional strong HKT manifolds.

\begin{theorem}\label{thm:main}
Let $(M,I,J,K,g)$ be a compact, simply connected, $8$-dimensional strong HKT manifold which is not hyper-K\"ahler. Then:
\begin{enumerate}[label=\arabic*.]
\item  the distribution spanned by $\{V, IV, JV, KV\}$ is involutive, and its Lie algebra is isomorphic to $\mathfrak{u}(1)\oplus\mathfrak{su}(2)$;
\item the soliton potential $f$ is constant.
\item the vector fields $IV$, $JV$ and $KV$ satisfy
\[
\mathcal{L}_{IV}\omega_J = -\omega_K,\qquad
\mathcal{L}_{JV}\omega_K = -\omega_I,\qquad
\mathcal{L}_{KV}\omega_I = -\omega_J;
\]

\end{enumerate}
\end{theorem}

\begin{proof}
$1.$ We first prove  $1.$  By statement~$4)$ in Remark~\ref{rmk:useful} we have
\[
d\theta = 2\, dV^\flat, \qquad 
\iota_V d\theta = \iota_{IV} d\theta = \iota_{JV} d\theta = \iota_{KV} d\theta = 0.
\]
Therefore, the vector fields $\{V, IV, JV, KV\}$ span a rank--$4$ distribution $\mathcal{F}\subset \ker(d\theta)$. We consider the orthogonal decomposition
\[
TM = \mathcal{F} \oplus \mathcal{F}^\perp
\]
with respect to the metric $g$.  
In the following, a tensor field $T$ will be called \emph{horizontal} if $\iota_X T = 0$ for all $X\in \mathcal{F}$.

We now show that $\mathcal{F}$ is involutive. Let $U,W$ be vector fields in $\mathcal{F}$ and decompose their Lie bracket as
\[
[U,W] = Z + X,
\]
where $Z\in \mathcal{F}$ and $X\in \mathcal{F}^\perp$. We claim that $X\equiv 0$.

Suppose by contradiction that there exists a point $p\in M$ such that $X_p\neq 0$. Since $U$ and $W$ belong to $\ker(d\theta)$ and $\ker(d\theta)$ is involutive, it follows that $[U,W]_p\in \ker(d\theta)_p$. As $Z_p\in \mathcal{F}_p\subset \ker(d\theta)_p$, we deduce that
\[
X_p = [U,W]_p - Z_p \in \ker(d\theta)_p.
\]
Thus $X_p$ lies in $\ker(d\theta)_p\cap \mathcal{F}^\perp_p$. Since $X_p$ is linearly independent from $\{V_p,IV_p,JV_p,KV_p\}$, this implies that $d\theta_p=0$.

By smoothness, there exists an open neighborhood $\mathcal{U}$ of $p$ such that $X_q\neq 0$ for all $q\in \mathcal{U}$. Consequently, $d\theta$ vanishes on $\mathcal{U}$, contradicting Proposition~\ref{prop:dtheta}. Hence $X\equiv 0$, and $\mathcal{F}$ is involutive.
\par
Therefore,  by Theorem \ref{prop:new} the Lie subalgebra of $\Gamma(TM)$ spanned by $\{V, IV, JV, KV\}$ is isomorphic to $\mathbb{R}^4$ if $a=0$, and to $\mathfrak{u}(1)\oplus\mathfrak{su}(2)$ otherwise, where $a$ is the constant defined in  \eqref{def:a}.

Using the same argument employed to compute
\[
d(JV^\flat) = -a\, KV^\flat \wedge IV^\flat + \beta_J,
\]
one similarly obtains
\[
d(KV^\flat) = -a\, IV^\flat \wedge JV^\flat + \beta_K,
\qquad
d(IV^\flat) = -a\, JV^\flat \wedge KV^\flat + \beta_I,
\]
where $\beta_I,\beta_J,\beta_K$ are horizontal $2$-forms. Moreover, since $d(IV^\flat)$, $d(JV^\flat)$ and $d(KV^\flat)$ are of type $(1,1)$ with respect to $I$, $J$ and $K$, respectively (see statements~$5)$--$7)$ of Remark~\ref{rmk:useful}), the forms $\beta_I$, $\beta_J$ and $\beta_K$ inherit the same property. 
For instance,
\[
\beta_I = d(IV^\flat) + a\, JV^\flat \wedge KV^\flat,
\]
and both terms on the right-hand side are of type $(1,1)$ with respect to $I$.

Since $V,IV,JV,KV$ form an orthonormal basis of $\mathcal{F}$, we may decompose the K\"ahler forms as
\[
\begin{aligned}
\omega_I &= V^\flat \wedge IV^\flat + JV^\flat \wedge KV^\flat + \omega_I^T, \\
\omega_J &= V^\flat \wedge JV^\flat + KV^\flat \wedge IV^\flat + \omega_J^T, \\
\omega_K &= V^\flat \wedge KV^\flat + IV^\flat \wedge JV^\flat + \omega_K^T,
\end{aligned}
\]
where $\omega_I^T,\omega_J^T,\omega_K^T$ are horizontal.

By Cartan’s formula,
\[
\mathcal{L}_{IV}\omega_J
= d(\iota_{IV}\omega_J) + \iota_{IV}(d\omega_J).
\]
For any vector field $X$,
\[
\iota_{IV}\omega_J(X)
= \omega_J(IV,X)
= g(JIV,X)
= - KV^\flat(X).
\]
Moreover,
\[
\begin{aligned}
\iota_{IV} d\omega_J(X,Y)
&= (\iota_{IV} JH)(X,Y)
= JH(IV,X,Y) \\
&= H(-JIV,JX,JY)
= H(KV,JX,JY) \\
&= H(KV,IX,IY)+H(V,IX,JY)+H(V,JX,IY),
\end{aligned}
\]
where in the last line we used Lemma~\ref{lem:useful}. Using the identities in \eqref{eqn:diff}, we obtain
\[
\begin{aligned}
\iota_{IV} d\omega_J(X,Y)
&= d(KV^\flat)(IX,IY)
 + dV^\flat(IX,JY)
 + dV^\flat(JX,IY) \\
&= d(KV^\flat)(IX,IY),
\end{aligned}
\]
where the last equality follows from the $\mathrm{SU}(2)$-invariance of
$dV^\flat=\tfrac12 d\theta$.
Therefore,
\begin{equation}\label{eqn:cartan}
\mathcal{L}_{IV}\omega_J
= I d(KV^\flat) - d(KV^\flat)
= a\,(V^\flat \wedge KV^\flat + IV^\flat \wedge JV^\flat)
  + I\beta_K - \beta_K.
\end{equation}
On the other hand, using the decomposition of $\omega_J$, we have
\begin{equation}\label{eqn:cartan2}
\mathcal{L}_{IV}\omega_J
= a\,(V^\flat \wedge KV^\flat + IV^\flat \wedge JV^\flat)
  + \mathcal{L}_{IV}\omega_J^T.
\end{equation}

Comparing \eqref{eqn:cartan} and \eqref{eqn:cartan2}, we obtain
\[
\mathcal{L}_{IV}\omega_J^T
= \iota_{IV} d\omega_J^T
= I\beta_K - \beta_K.
\]
Since $\beta_K$, and hence $I\beta_K$, are $(1,1)$-forms with respect to $K$, it follows that $\mathcal{L}_{IV}\omega_J^T$ is also of type $(1,1)$ with respect to $K$.

The horizontal distribution $\mathcal{F}^\perp$ has dimension $4$. Consequently, the bundle of horizontal $2$-forms decomposes as
\[
\bigwedge\nolimits^2 \mathcal{F}^\perp
= \bigwedge\nolimits^+ \mathcal{F}^\perp
\oplus
\bigwedge\nolimits^- \mathcal{F}^\perp,
\]
where the splitting is taken with respect to the horizontal metric $g^T$. It is straightforward to check that
\[
\bigwedge\nolimits^+ \mathcal{F}^\perp
= \mathrm{span}\{\omega_I^T,\omega_J^T,\omega_K^T\},
\]
whereas $\bigwedge\nolimits^- \mathcal{F}^\perp$ coincides with the space of horizontal $\mathrm{SU}(2)$-invariant $2$-forms.

In particular, the $2$-form $dV^\flat$ belongs to $\bigwedge^- \mathcal{F}^\perp$, since it is horizontal and $\mathrm{SU}(2)$-invariant.

Moreover, by Remark~\ref{rmk:useful}, statement~$1)$, the vector fields $V,IV,JV,KV$ are Killing. Hence, the corresponding flows preserve the Riemannian volume and the Hodge star operator associated with $g^T$. It follows that $\mathcal{L}_{IV}\omega_J^T$ is self-dual, and therefore belongs to $\bigwedge^+ \mathcal{F}^\perp$.

Since $\mathcal{L}_{IV}\omega_J^T$ is both self-dual and of type $(1,1)$ with respect to $K$, there exists a smooth function $b_{IJ}\in C^\infty(M)$ such that
\[
\mathcal{L}_{IV}\omega_J^T = b_{IJ}\,\omega_K^T.
\]

By analogous computations,  up to a cyclic permutation of the complex structures,  we obtain
\begin{equation}\label{eqn:beta}
\begin{aligned}
\mathcal{L}_{IV}\omega_J^T &= \iota_{IV} d\omega_J^T = I\beta_K - \beta_K = b_{IJ}\,\omega_K^T,\\
\mathcal{L}_{IV}\omega_K^T &= \iota_{IV} d\omega_K^T = -I\beta_J + \beta_J = b_{IK}\,\omega_J^T,\\
\mathcal{L}_{JV}\omega_I^T &= \iota_{JV} d\omega_I^T = -J\beta_K + \beta_K = b_{JI}\,\omega_K^T,\\
\mathcal{L}_{JV}\omega_K^T &= \iota_{JV} d\omega_K^T = J\beta_I - \beta_I = b_{JK}\,\omega_I^T,\\
\mathcal{L}_{KV}\omega_I^T &= \iota_{KV} d\omega_I^T = K\beta_J - \beta_J = b_{KI}\,\omega_J^T,\\
\mathcal{L}_{KV}\omega_J^T &= \iota_{KV} d\omega_J^T = -K\beta_I + \beta_I = b_{KJ}\,\omega_I^T.
\end{aligned}
\end{equation}
Since $\beta_I$, $\beta_J$ and $\beta_K$ are of type $(1,1)$ with respect to $I$, $J$ and $K$, respectively, we may decompose them as
\[
\beta_I = \lambda_I \,\omega_I^T + \eta_I,
\qquad
\beta_J = \lambda_J \,\omega_J^T + \eta_J,
\qquad
\beta_K = \lambda_K \,\omega_K^T + \eta_K,
\]
where $\eta_I,\eta_J,\eta_K$ are anti-self-dual horizontal $2$-forms and
$\lambda_I,\lambda_J,\lambda_K\in C^\infty(M)$.

We first show that
\[
b_{IJ}=b_{JK}=b_{KI}=b_{JI}=b_{KJ}=b_{IK}.
\]
Since the forms $\eta_L$ are $\mathrm{SU}(2)$-invariant, we compute
\begin{equation}\label{eqn:lambda}
I\beta_K - \beta_K
= -\lambda_K \omega_K^T + \eta_K - \lambda_K \omega_K^T - \eta_K
= -2\lambda_K \omega_K^T
= b_{IJ}\,\omega_K^T.
\end{equation}
Similarly,
\[
- J\beta_K + \beta_K
= 2\lambda_K \omega_K^T
= b_{JI}\,\omega_K^T,
\]
and hence $b_{IJ} = - b_{JI}$. By the same argument, one obtains
\[
b_{JK} = - b_{KJ},
\qquad
b_{KI} = - b_{IK}.
\]

On the other hand, for any horizontal vector fields $X,Y$ we have
\[
H(IV,X,Y)
= H(IV,JX,JY) + H(JIV,JX,Y) + H(JIV,X,JY),
\]
and therefore, using \eqref{eqn:diff},
\[
(\beta_I - J\beta_I)(X,Y)
= H(IV,X,Y) - H(IV,JX,JY)
= (J\beta_K - \beta_K)(X,JY).
\]
Contracting both sides of \eqref{eqn:beta} with $X$ and $Y$, we obtain
\[
-b_{JK}\, g(IX,Y)
= -b_{JK}\, \omega_I^T(X,Y)
= b_{IJ}\, \omega_K^T(X,JY)
= b_{IJ}\, g(KX,JY)
= -b_{IJ}\, g(IX,Y),
\]
which implies $b_{IJ}=b_{JK}$. By cyclic permutation, we also obtain
$b_{JK}=b_{KI}$.

Moreover, from \eqref{eqn:lambda} we deduce
\[
\lambda_I = \lambda_J = \lambda_K = -\frac{b}{2},
\]
where
\[
b := b_{IJ}=b_{JK}=b_{KI}=b_{JI}=b_{KJ}=b_{IK}.
\]

Summarizing, we obtain
\begin{equation}\label{eqn:beta1}
\begin{aligned}
\mathcal{L}_{IV}\omega_J^T &= \iota_{IV} d\omega_J^T = I\beta_K - \beta_K = b\,\omega_K^T,\\
\mathcal{L}_{IV}\omega_K^T &= \iota_{IV} d\omega_K^T = -I\beta_J + \beta_J = -b\,\omega_J^T,\\
\mathcal{L}_{JV}\omega_I^T &= \iota_{JV} d\omega_I^T = -J\beta_K + \beta_K = -b\,\omega_K^T,\\
\mathcal{L}_{JV}\omega_K^T &= \iota_{JV} d\omega_K^T = J\beta_I - \beta_I = b\,\omega_I^T,\\
\mathcal{L}_{KV}\omega_I^T &= \iota_{KV} d\omega_I^T = K\beta_J - \beta_J = b\,\omega_J^T,\\
\mathcal{L}_{KV}\omega_J^T &= \iota_{KV} d\omega_J^T = -K\beta_I + \beta_I = -b\,\omega_I^T.
\end{aligned}
\end{equation}

By Remark~\ref{rmk:useful}, statements 1), 2) and 3), we have
\[
\mathcal{L}_{V}\omega_I
=\mathcal{L}_{V}\omega_J
=\mathcal{L}_{V}\omega_K
=\mathcal{L}_{IV}\omega_I
=\mathcal{L}_{JV}\omega_J
=\mathcal{L}_{KV}\omega_K
=0.
\]
In particular, this implies
\[
\iota_V(d\omega_I^T)=\iota_{IV}(d\omega_I^T)=0,
\]
and analogous identities hold for $\omega_J^T$ and $\omega_K^T$.

Indeed, since $V,IV,JV,KV$ are Killing vector fields, we compute
\[
\begin{aligned}
0
&=\mathcal{L}_{V}\omega_I
=\mathcal{L}_{V}(V^\flat \wedge IV^\flat + JV^\flat \wedge KV^\flat + \omega_I^T)
=\mathcal{L}_{V}(\omega_I^T)
=\iota_V(d\omega_I^T),\\
0
&=\mathcal{L}_{IV}\omega_I
=\mathcal{L}_{IV}(V^\flat \wedge IV^\flat + JV^\flat \wedge KV^\flat + \omega_I^T)
=\mathcal{L}_{IV}(\omega_I^T)
=\iota_{IV}(d\omega_I^T).
\end{aligned}
\]

Combining these properties with equations~\eqref{eqn:beta1}, we deduce that
\[
\begin{aligned}
d\omega_I^T
&= b\, KV^\flat \wedge \omega_J^T - b\, JV^\flat \wedge \omega_K^T + \gamma_I,\\
d\omega_J^T
&= b\, IV^\flat \wedge \omega_K^T - b\, KV^\flat \wedge \omega_I^T + \gamma_J,\\
d\omega_K^T
&= b\, JV^\flat \wedge \omega_I^T - b\, IV^\flat \wedge \omega_J^T + \gamma_K,
\end{aligned}
\]
where $\gamma_I,\gamma_J,\gamma_K$ denote the purely horizontal components of
$d\omega_I^T,d\omega_J^T$ and $d\omega_K^T$, respectively.

Since the horizontal distribution $\mathcal{F}^\perp$ has dimension $4$, each
$\gamma_L$ must be of the form
\[
\gamma_L = \theta_L^T \wedge \omega_L^T,
\]
where $\theta_L^T$ is the transverse Lee form associated with $\omega_L^T$,
for $L\in\{I,J,K\}$.

We now compute the transverse Lee forms $\theta_I^T$, $\theta_J^T$ and
$\theta_K^T$. Recall that the Lee form is given by
\begin{equation}\label{eqn:theta}
\theta_I(X)=\theta(X)=\frac12\sum_{i=1}^8 H(e_i,Ie_i,IX),
\end{equation}
see for instance~\cite[Lemma~8.6]{GFS}, where $\{e_1,\dots,e_8\}$ is any local
orthonormal frame.

Without loss of generality, we choose the local orthonormal frame
\[
\{V,IV,JV,KV,\xi_1,\xi_2,\xi_3,\xi_4\},
\]
where $\{\xi_1,\dots,\xi_4\}$ is a local horizontal orthonormal frame adapted to
the hypercomplex structure, namely
\[
\xi_2=I\xi_1,\qquad \xi_3=J\xi_1,\qquad \xi_4=K\xi_1.
\]

We claim that
\begin{equation}\label{eqn:thet}
\theta_I=-(a+b)\,V^\flat+\theta_I^T.
\end{equation}

Indeed, using~\eqref{eqn:theta}, we compute
\[
\begin{aligned}
\theta_I(V)
&=H(V,IV,IV)+H(JV,KV,IV)
  +H(\xi_1,\xi_2,IV)+H(\xi_3,\xi_4,IV)\\
&=-a+d(IV^\flat)(\xi_1,\xi_2)+d(IV^\flat)(\xi_3,\xi_4)\\
&=-a+d(IV^\flat)(\xi_1,I\xi_1)+d(IV^\flat)(J\xi_1,K\xi_1)\\
&=-a+\bigl(dIV^\flat-JdIV^\flat\bigr)(\xi_1,I\xi_1)\\
&=-a+(\beta_I-J\beta_I)(\xi_1,I\xi_1)\\
&=-a-b\,\omega_I^T(\xi_1,I\xi_1)\\
&=-(a+b),
\end{aligned}
\]
where we used~\eqref{eqn:diff} and~\eqref{eqn:beta1}.

Similarly,
\[
\begin{aligned}
\theta_I(IV)
&=-H(V,IV,V)-H(JV,KV,V)
  -H(\xi_1,\xi_2,V)-H(\xi_3,\xi_4,V)\\
&=-(dV^\flat-JdV^\flat)(\xi_1,I\xi_1)=0,
\end{aligned}
\]
which vanishes since $dV^\flat$ is $(1,1)$ with respect to $I,J,K$, and
\[
\begin{aligned}
\theta_I(JV)
&=H(V,IV,KV)+H(JV,KV,KV)
  +H(\xi_1,\xi_2,KV)+H(\xi_3,\xi_4,KV)\\
&=(dKV^\flat-JdKV^\flat)(\xi_1,I\xi_1)\\
&=(\beta_K-J\beta_K)(\xi_1,I\xi_1)\\
&=-b\,\omega_K^T(\xi_1,I\xi_1)=0.
\end{aligned}
\]
An analogous computation shows that $\theta_I(KV)=0$.

Finally, for any $X\in\mathcal{F}^\perp$, we obtain
\[
\begin{aligned}
\theta_I(X)
&=H(V,IV,IX)+H(JV,KV,IX)
  +H(\xi_1,\xi_2,IX)+H(\xi_3,\xi_4,IX)\\
&=H^T(\xi_1,\xi_2,IX)+H^T(\xi_3,\xi_4,IX)\\
&=\theta_I^T(X),
\end{aligned}
\]
where $H^T$ denotes the Bismut torsion of the transverse hyper-Hermitian
structure. This proves~\eqref{eqn:thet}.

By the definition of $V$ and using~\eqref{eqn:thet}, we have
\[
V^\flat=\tfrac12(\theta-df)
= -\tfrac12(a+b)\,V^\flat+\tfrac12(\theta_I^T-df),
\]
which immediately yields
\begin{equation} \label{eqn:a+b}
a+b=-2\qquad\text{and}\qquad \theta_I^T=df.
\end{equation}
Here we used that $df$ is horizontal (see Remark~\ref{rmk:useful}, statement~$8)$).
In particular, $b$ must be constant. By the same argument we also obtain
\[
\theta_J^T=df\qquad\text{and}\qquad \theta_K^T=df.
\]

Summarizing, the transverse fundamental forms satisfy
\begin{equation}\label{eqn:domega}
\begin{aligned}
d\omega_I^T&= b\,KV^\flat\wedge\omega_J^T
             -b\,JV^\flat\wedge\omega_K^T
             +df\wedge\omega_I^T,\\
d\omega_J^T&= b\,IV^\flat\wedge\omega_K^T
             -b\,KV^\flat\wedge\omega_I^T
             +df\wedge\omega_J^T,\\
d\omega_K^T&= b\,JV^\flat\wedge\omega_I^T
             -b\,IV^\flat\wedge\omega_J^T
             +df\wedge\omega_K^T.
\end{aligned}
\end{equation}

Taking the exterior differential of the first equation in~\eqref{eqn:domega}
and using $d^2=0$, we obtain the algebraic condition
\begin{equation}\label{productab}
b(b-a)=0.
\end{equation}
If $a=0$, then $b=0$, which contradicts the relation $a+b=-2$.
Therefore $a\neq0$, proving the first statement of the theorem.

\smallskip

$2.$ We now prove  $2.$.
Using the decomposition
\[
\omega_I = V^\flat \wedge IV^\flat + JV^\flat \wedge KV^\flat + \omega_I^T,
\]
a direct computation yields
\begin{equation}\label{eqn:torsion}
\begin{split}
H
&= V^\flat \wedge dV^\flat
- a\, IV^\flat \wedge JV^\flat \wedge KV^\flat
+ IV^\flat \wedge \beta_I
+ JV^\flat \wedge \beta_J
+ KV^\flat \wedge \beta_K
- Idf \wedge \omega_I^T \\
&= V^\flat \wedge dV^\flat
- a\, IV^\flat \wedge JV^\flat \wedge KV^\flat
+ IV^\flat \wedge \beta_I
+ JV^\flat \wedge \beta_J
+ KV^\flat \wedge \beta_K
+ H^T,
\end{split}
\end{equation}
where $H^T=-Idf \wedge \omega_I^T$ denotes the torsion of the transverse HKT structure.  Note that  $H$ is written as a sum of a Chern-Simons $3$-form and the torsion on the space of orbits.

Let $X,Y,Z,W$ be horizontal vector fields. From the expression above we obtain
\begin{equation} \label{eqn:dH}
\begin{split}
dH(X,Y,Z,W)
&= dV^\flat \wedge dV^\flat (X,Y,Z,W)
+ \beta_I \wedge \beta_I (X,Y,Z,W) \\
&\quad + \beta_J \wedge \beta_J (X,Y,Z,W)
+ \beta_K \wedge \beta_K (X,Y,Z,W)
+ dH^T(X,Y,Z,W). \\
\end{split}
\end{equation}
Using the decomposition
\[
\beta_L = -\frac{b}{2}\,\omega_L^T + \eta_L, \qquad L \in \{I,J,K\},
\]
where $\eta_L$ are anti-self-dual horizontal $2$-forms, we obtain
\begin{equation}\label{eqn:dH-horizontal}
\begin{split}
dH(X,Y,Z,W)
&=\left(
-\frac{1}{2}\bigl(\|dV^\flat\|^2
+ \|\eta_I\|^2
+ \|\eta_J\|^2
+ \|\eta_K\|^2\bigr)
+ \frac{3}{2} b^2
\right)\, \mathrm{vol}^T \\
&\quad + dH^T(X,Y,Z,W),
\end{split}
\end{equation}
where $\mathrm{vol}^T$ denotes the volume form of the transverse hyper-Hermitian structure $(\omega_I^T, \omega_J^T, \omega_K^T, g^T)$.

By \cite[Proposition~4.33]{GFS}, one has \footnote{This equation is also known in physics as 
 dilaton field equation, see for instance \cite{IS-24}.}
\begin{equation}\label{eqn:gf}
\frac{1}{6}\|H\|^2 + \Delta f - \|\operatorname{grad} f\|^2 = C,
\end{equation}
for some constant $C \in \mathbb{R}$.

Using \eqref{eqn:torsion} and the fact that, for any $L=I,J,K$,
\[
\|\beta_L\|^2
= \frac{b^2}{4}\|\omega_L^T\|^2 + \|\eta_L\|^2
= b^2 + \|\eta_L\|^2,
\]
we obtain
\[
\|H\|^2
= \|H^T\|^2
+ 3\bigl(\|dV^\flat\|^2
+ \|\eta_I\|^2
+ \|\eta_J\|^2
+ \|\eta_K\|^2\bigr)
+ C_0,
\]
where $C_0$ is a constant. Moreover,\footnote{This computation already appeared in an early version of \cite{ALS}.}
\[
\|H^T\|^2 = 6 \|\operatorname{grad} f\|^2 .
\]

Plugging this into \eqref{eqn:gf}, we obtain
\begin{equation}\label{eqn:laplace}
\frac{1}{2}\bigl(
\|dV^\flat\|^2
+ \|\eta_I\|^2
+ \|\eta_J\|^2
+ \|\eta_K\|^2
\bigr)
+ \Delta f
= C_1,
\end{equation}
for some constant $C_1 \in \mathbb{R}$.

Let $\{e_1,e_2,e_3,e_4\}$ be a local $I$-unitary orthonormal horizontal frame.
Evaluating $dH$ on $(e_1,e_2,e_3,e_4)$ and using \eqref{eqn:dH-horizontal}, we obtain
\[
\begin{split}
dH(e_1,e_2,e_3,e_4)
=& -\frac{1}{2}\bigl(
\|dV^\flat\|^2
+ \|\eta_I\|^2
+ \|\eta_J\|^2
+ \|\eta_K\|^2
\bigr)
+ \frac{3}{2}b^2 \\
&\quad - d(Idf) \wedge \omega_I^T(e_1,e_2,e_3,e_4)
+ Idf \wedge df \wedge \omega_I^T(e_1,e_2,e_3,e_4).
\end{split}
\]

We now compute the second term. Using that $\omega_I^T(e_1,e_2)=\omega_I^T(e_3,e_4)=1$, we have
\[
\begin{split}
d(Idf) \wedge \omega_I^T(e_1,e_2,e_3,e_4)
=&\ d(Idf)(e_1,e_2) + d(Idf)(e_3,e_4) \\
=& \sum_{j=1}^4 \mathcal{L}_{e_j}(df(e_j))
- \sum_{i=1}^2 Idf([e_{2i-1},e_{2i}]) \\
=& \sum_{j=1}^4 \mathcal{L}_{e_j}(df(e_j))
- \sum_{i=1}^2 Idf\!\left(
\nabla^{\mathrm{Ob}}_{e_{2i-1}}e_{2i}
- \nabla^{\mathrm{Ob}}_{e_{2i}}e_{2i-1}
\right) \\
=& \sum_{j=1}^4 \mathcal{L}_{e_j}(df(e_j))
- df(\nabla^{\mathrm{Ob}}_{e_j}e_j) \\
=& \sum_{j=1}^4 \mathcal{L}_{e_j}(df(e_j))
- df(\nabla^{\mathrm{B}}_{e_j}e_j)
- A(e_j,e_j,\operatorname{grad} f),
\end{split}
\]
where in the last equality we used \eqref{eqn:obata}.\par
We observe that, since $e_j$ and $\operatorname{grad} f$ are horizontal vector fields, equation \eqref{eqn:A} yields
\[
\begin{split}
2A(e_j,e_j,\operatorname{grad} f)
=& -H(e_j,Ie_j,I\operatorname{grad} f) - H(Ie_j,Ie_j,\operatorname{grad} f) \\
& - H(e_j,Ke_j,K\operatorname{grad} f) - H(Ie_j,Ke_j,J\operatorname{grad} f).
\end{split}
\]
Since all arguments are horizontal, the torsion $H$ may be replaced by the transverse torsion $H^T$, and therefore
\[
\begin{split}
2A(e_j,e_j,\operatorname{grad} f)
=& -H^T(e_j,Ie_j,I\operatorname{grad} f)
- H^T(e_j,Ke_j,K\operatorname{grad} f) \\
& - H^T(Ie_j,Ke_j,J\operatorname{grad} f).
\end{split}
\]

A straightforward computation then shows that
\[
\begin{split}
\sum_{j=1}^4 A(e_j,e_j,\operatorname{grad} f)
=& -\frac{1}{2} \sum_{j=1}^4
\Big(
H^T(e_j,Ie_j,I\operatorname{grad} f)
+ H^T(e_j,Ke_j,K\operatorname{grad} f)
\Big) \\
& + \frac{1}{2} \sum_{j=1}^4 H^T(e_j,Je_j,J\operatorname{grad} f).
\end{split}
\]
Recalling the definition of the transverse Lee form $\theta^T$, we obtain
\[
\sum_{j=1}^4 A(e_j,e_j,\operatorname{grad} f)
= -\theta^T(\operatorname{grad} f)
= -\|\operatorname{grad} f\|^2,
\]
where in the last equality we used that $\theta^T=df$.

Therefore,
\[
\begin{split}
(dIdf \wedge \omega_I^T)(e_1,e_2,e_3,e_4)
=& \sum_{j=1}^4 \mathcal{L}_{e_j}(df(e_j))
- df(\nabla^{LC}_{e_j} e_j)
+ \|\operatorname{grad} f\|^2 \\
=& \Delta f + \|\operatorname{grad} f\|^2 .
\end{split}
\]
A direct computation also shows that
\[
(Idf \wedge df \wedge \omega_I^T)(e_1,e_2,e_3,e_4)
= -\|\operatorname{grad} f\|^2.
\]

We then finally obtain
\[
dH(e_1,e_2,e_3,e_4)
= -\frac{1}{2}(\|dV^\flat\|^2 + \|\eta_I\|^2 + \|\eta_J\|^2 + \|\eta_K\|^2)+ \frac{3}{2} b^2
- \Delta f - 2 \|\operatorname{grad} f\|^2 = 0.
\]
Since $-\frac{1}{2} (\|dV^\flat\|^2 + \|\eta_I\|^2 + \|\eta_J\|^2 + \|\eta_K\|^2) - \Delta f$ is constant, it follows that
\[
2\|\operatorname{grad} f\|^2 = \mathrm{const}.
\]
By the maximum principle, this forces $f$ to be constant, proving the third statement of the theorem.

\smallskip

$3.$   Finally,  we  prove $3.$  Using the proof of  1., in particular using \eqref{eqn:a+b} and \eqref{productab}, we are  left with two possibilities:
\[
(a,b)=(-1,-1)\qquad\text{or}\qquad (a,b)=(-2,0).
\]
We exclude the case $b=0$ and $a=-2$. 
If $b=0$, then $dH=0$ and equation \eqref{eqn:dH-horizontal} imply
\[
\|dV^\flat\|^2 + \|\eta_I\|^2 + \|\eta_J\|^2 + \|\eta_K\|^2 = 0.
\]
Consequently $dV^\flat=0$, and hence $d\theta=0$, since $2dV^\flat=d\theta$.
This contradicts Proposition~\ref{prop:dtheta}. \par

It remains to consider $a=b=-1$. The equations~ \eqref{eqn:cartan2} and \eqref{eqn:beta1} immediately imply
\[
\mathcal{L}_{IV}\omega_J=-\omega_K.
\]
The other identities 
\[
\mathcal{L}_{JV}\omega_K=-\omega_I,\qquad
\mathcal{L}_{KV}\omega_I=-\omega_J,
\]
follow similarly, using equations \eqref{eqn:beta1} and the bracket relations between $V,IV,JV,KV$.  
\end{proof}

{\begin{remark} We point out that the dimensional assumption  in the proof of  Theorem \ref{thm:main} is essential. Indeed, the fact that the manifold is $8$-dimensional is used to show that, if $X$ is non zero on $\mathcal{U}$, then $d\theta$ would vanish on $\mathcal{U}$, since the vector fields \[ V,\, IV,\, JV,\, KV,\, X,\, IX,\, JX,\, KX \] would span the kernel of $d\theta$ on $\mathcal{U}$. Moreover, in order to apply Proposition~\ref{prop:dtheta} and derive a contradiction, the simply connectedness assumption is also required. A  result in any dimension  has been  recently proved  in \cite{KS-26}.
\end{remark}

By statements $1.$ and $3.$ of the previous Theorem, we may immediately prove the following
\begin{corollary} \label{cor:properties}
Let $(M,I,J,K,g)$ be a strong HKT compact and simply connected $8$-dimensional manifold which is not hyper-K\"ahler. Then the following statements holds true
\begin{enumerate}[label=\arabic*.]
\item[1)]  $\mathbb{H}^*$ acts on $M$ isometrically.
\item[2)] $V,IV,JV,KV$ are infinitesimal automorphisms preserving $\cal{F}^\perp \in \operatorname{TM}$, namely $[W,\cal{F}^\perp] \subset \cal{F}^\perp$ for any $W \in \{V,IV,JV,KV\}$, where $\cal{F}$ is the distribution generated by $V,IV,JV,KV$ and $\cal{F}^\perp$ is its orthogonal complement with respect to $g$.
\item[3)] $\nabla^{Ob} V =\frac{1}{2} Id, \ \nabla^{Ob} IV =\frac{1}{2} I, \ \nabla^{Ob} JV =\frac{1}{2} J, \ \nabla^{Ob} KV =\frac{1}{2} K$.
\item[4)]  All the components of the Obata curvature tensor involving the quaternionic span of $V$ vanishes.
\end{enumerate}
\end{corollary}
\begin{proof}
By the previous Theorem the Lie algebra of the distribution spanned by $\{V, I V, J V, KV\}$ is isomorphic to $\mathfrak{u}(1) \oplus \mathfrak{su}(2)$. Furthermore, since $V,IV,JV,KV$ are Killing (see Remark \ref{rmk:useful}), they generate an isometric action of $\mathbb{H}^*$ on $M$. This concludes the proof of the first claim. \par

For the second claim we only have to check that $[W,\cal{F}^\perp] \subset \cal{F}^\perp$ for any $W \in \{V,IV,JV,KV\}$. Let us fix the following notation, let $V=I_0 V$, $IV=I_1 V$, $JV=I_2 V$ and $KV=I_3 V$. Using that $I_a V$ is Killing for $a=0,\dots,3$ we get that for any $a,b=0,\dots,3$ and $X \in \cal{F}^\perp$
\[
0=\cal{L}_{I_{a}V} g (X, I_b V)=g([I_a V,X],I_b V),
\]
forcing $[I_a V,X] \in \cal{F}^\perp$.\par

The third and fourth claim are consequences of \cite[Lemma 2.2 and Proposition 2.3]{PPS}, where the authors proved that if we have a hyper-complex manifold endowed with a vector field $W$ such that $W, IW,JW,KW$ span a copy of $\mathrm{U}(2)$ where $W$ is in the center $\mathrm{U}(1)$, then if $W$ is hyper-holomorphic and $\cal{L}_{IW}J=K$,  $\cal{L}_{JW}K=I$, $\cal{L}_{KW}I=J$, this leads to $\nabla^{Ob} W=-\frac{1}{2} Id$. Applying this to $-V$ we immediately get the statements $3)$ and  $4)$.
\end{proof}

\begin{remark} Note that if the distribution generated by $V, IV, JV$ and $KV$ is integrable, then the statement 2) of the above corollary holds in any dimension. 
\end{remark}

The vector field $V$ is commonly referred to as the \emph{Euler} vector field, and the corresponding hyper-complex structure is described as \emph{conical} \cite{CH2} ( see Section \ref{Potential} for more details).
\begin{remark}
Under the hypothesis of Theorem \ref{thm:main}, as a consequence of the previous Proposition we get that the orbits of the foliation $\cal{F}$ are totally geodesic and Obata flat. Furthermore, as a consequence of the fact that $\nabla^{Ob} V=\frac{1}{2} Id$, we observe that the Obata connection can only preserve tensors of the kind $(k,k)$, as pointed out in \cite{So}. 
\end{remark}

\section{ Existence of a foliation  with compact Leaves} \label{section6}

 Let $(M, I,J,K,g)$  be a compact, $8$ dimensional simply connected strong HKT manifold that is not hyper-K\"ahler. Denote by $\mathrm{Iso}(M)$ its isometry group. By the Myers--Steenrod theorem, $\mathrm{Iso}(M)$ is a compact Lie group. \par
By Corollary \ref{cor:properties}, we have  there exists an isometric action of $\mathbb{H}^*$ on $M$. Let
\[
\rho \colon \mathbb{H}^*\cong SU(2)\times \mathbb{R} \longrightarrow \mathrm{Iso}(M),
\]
be the associated Lie group homomorphism and let $H:=\mathrm{Im}(\rho)$. Since $\mathrm{Iso}(M)$ is compact, the closure $K:=\overline{H}$ is a compact Lie subgroup of $\mathrm{Iso}(M)$. Its Lie algebra is of the form
\[
\mathfrak{k}\simeq \mathfrak{su}(2)\oplus \mathbb{R}^m,
\]
where the abelian factor corresponds to the closure of the $\mathbb{R}$-action.

Let $G:=\overline{\rho(\mathbb{R})}\subset \mathrm{Iso}(M)$. Then $G$ is a compact torus, hence a compact abelian Lie group, acting smoothly, isometrically  and hyper-holomorphically on $M$. Since $G$ is compact, the action is proper, and since $M$ is compact, all the hypotheses of the isotropy type stratification theorem are satisfied. For each $x\in M$, denote by
\[
G_x:=\{g\in G\mid g\cdot x=x\}
\]
the isotropy subgroup at $x$. Define the set of isotropy types by
\[
\mathcal O(M,G):=\{[G_x]\mid x\in M\},
\]
where $[G_x]$ denotes the conjugacy class of $G_x$ in $G$. Since $G$ is abelian, conjugation is trivial and therefore each isotropy type coincides with the isotropy subgroup itself. By the general theory of proper actions of compact Lie groups on compact manifolds, the set $\mathcal O(M,G)$ is finite. In particular, only finitely many closed subgroups of $G$ occur as isotropy groups of points of $M$.

Let $T_1,\dots,T_r\subset G$ be the isotropy subgroups which are maximal with respect to inclusion. The finiteness of $\mathcal O(M,G)$ implies that $r<\infty$, and by maximality every isotropy subgroup $G_x$ is contained in one $T_i$. Let $\mathfrak t_i\subset\mathbb{R}^m$ denote the Lie algebra of $T_i$. \par
By hypothesis, there exists a vector $v \in\mathbb{R}^m$ with associated Killing field $V$. Since $V$ is nowhere vanishing, $v$ does not belong to the Lie algebra of the isotropy group at any point of $M$. In particular, $v \notin\mathfrak t_i$ for all $i=1,\dots,r$, and hence each $\mathfrak t_i$ is a proper Lie subalgebra of $\mathbb{R}^m$.

It follows that for any $v'$ such that $v' \notin \mathfrak{t}_i$, for each $i$, the associated Killing field $V'$ is nowhere vanishing on $M$.  Let us fix any of such $v'$. If furthermore $v'\in\mathbb{Z}^m$ (we can always make such a choice), then the flow of $V'$ has closed orbits. \par
Since $v' \in \mathbb{R}^m= Lie (G)$, $V'$ is Killing, hyper-holomorphic and commutes with $V$. In particular
\[
[V',IV]=[V',JV]=[V',KV]=0, \ \ \mathcal{L}_V' g=0, \ \ \mathcal{L}_V' d\theta=0,
\]
\[
\langle V', IV,JV,KV \rangle \cong \mathfrak{u}(1) \oplus \mathfrak{su}(2),
\]
and the corresponding action on $M$ has closed orbits. We then obtain the following
\begin{theorem}
Let $(M,I,J,K,g)$ be a compact, simply connected, $8$-dimensional strong HKT manifold.
Then either $M$ is hyper-K\"ahler, or $M$ admits the structure of a Hopf fibration over a $4$-dimensional oriented Riemannian orbifold with trivial orbifold fundamental group.
\end{theorem}
\begin{remark}
%\begin{enumerate}
 We point out, that the fibers are Hopf surfaces, but not  necessarily quaternionic. However, when $V$ has closed orbits itself, the fibers are hypercomplex. In this case it is well known that the quotient geometry is QKT - see \cite{Poon-Swann} for example. However, in real dimension $4$ a QKT structure is equivalent to an oriented Riemannian structure. In  particular, the fact that the fibers are hypercomplex does not provide any new information in our case and one can use $V'$ instead of $V$. We note that further results in the real dimension 8 are announced recently in \cite{P-25}.
%\end{enumerate}
\end{remark}

\section{Characterization of $\nabla^B H=0$} \label{section7}

Since it will be useful in the following, we summarize some facts from the proof of Theorem \ref{thm:main}: the expression of $H$ is given by
\begin{equation} \label{eqn:tor2}
H= V^\flat \wedge dV^\flat + IV^\flat \wedge JV^\flat \wedge KV^\flat + IV^\flat \wedge \beta_I + JV^\flat \wedge \beta_J + KV^\flat \wedge \beta_K,
\end{equation}
whereas the differential of $IV^\flat,JV^\flat,KV^\flat$ can be written as
\begin{equation} \label{eqn:diff2}
dIV^\flat=JV^\flat \wedge KV^\flat + \beta_I,\ \  dJV^\flat=KV^\flat \wedge IV^\flat +\beta_J, \ \ dKV^\flat=IV^\flat \wedge JV^\flat +\beta_K
\end{equation}
where 
\begin{equation} \label{eqn:betas}
\beta_I=\frac{1}{2} \omega_I^T + \eta_I, \ \beta_J= \frac{1}{2} \omega_J^T + \eta_J, \ \beta_K= \frac{1}{2} \omega_K^T + \eta_K,
\end{equation}
and  $\eta_I,\eta_J,\eta_K$ are anti self dual horizontal $2$-forms. Using equations \eqref{eqn:domega} and $d^2 IV^\flat=d^2JV^\flat= d^2KV^\flat=0$, we recover that
\begin{equation}\label{eqn:deta}
d\eta_I=- \eta_J \wedge KV^\flat + \eta_K \wedge JV^\flat, \ 
d\eta_J=- \eta_K \wedge IV^\flat +\eta_I \wedge KV^\flat, \ \ d\eta_K=- \eta_I \wedge JV^\flat + \eta_J \wedge IV^\flat.
\end{equation} 
Furthermore, the strong HKT condition implies that 
\begin{equation}\label{eqn:dH3}
0=dH= \left[ -\frac{1}{2}\left(\norm{dV^\flat}^2+ \norm{\eta_I}^2 + \norm{\eta_J}^2 + \norm{\eta_K}^2 \right) + \frac{3}{2}  \right] vol^T.
\end{equation} 
We recall that a HKT structure $(I,J,K,g)$ is said to be Einstein HKT if there exists a smooth function $\lambda$ such that 
\[
\frac{\rho^C_I - J \rho^C_I}{2}=\lambda \,  \omega_I,
\]
where $\rho^C_I$ is the (first) Chern Ricci curvature of $(I,g)$.
 
  \begin{proposition}
Any $8$-dimensional compact simply connected strong HKT manifold $(M,I,J,K,g)$ is HKT Einstein with constant $\lambda$.
\end{proposition}
\begin{proof}
If $(M,I,J,K,g)$ is hyper-K\"ahler, then it is HKT Einstein with $\lambda=0$. Let us assume then that $(M,I,J,K,g)$ is not hyper-K\"ahler. Since $\rho^B=0$, then $\rho^C_I=dI\theta=2 dIV^\flat$. Therefore
\[
\begin{split}
\frac{\rho^C_I - J \rho^C_I}{2}=&dIV^\flat- J (dIV^\flat)\\
=&JV^\flat \wedge KV^\flat + \frac{1}{2} \omega_I^T + \eta_I- J(JV^\flat \wedge KV^\flat + \frac{1}{2} \omega_I^T + \eta_I) \\
=&JV^\flat \wedge KV^\flat + \frac{1}{2} \omega_I^T + \eta_I + V^\flat \wedge IV^\flat+ \frac{1}{2} \omega_I^T-\eta_I \\
=&V^\flat \wedge IV^\flat+JV^\flat \wedge KV^\flat +\omega_I^T \\
=& \omega_I,
\end{split}
\]
concluding the proof.
\end{proof}
We determine a sufficient condition to have parallel Bismut torsion in terms of $\eta_I,\eta_J,\eta_K$.
\begin{theorem} \label{thm:eta}
Let $(M,I,J,K,g)$ be a strong HKT compact and simply connected $8$-dimensional manifold which is not hyper-K\"ahler. If any of $\eta_I, \eta_J, \eta_K$ is zero, then all of them are zero and $\nabla^B H=0$. In particular, $(M,I,J,K,g)$ is the hyper-Hermitian Samelson space $SU(3)$.
\end{theorem}
\begin{proof}
By \eqref{eqn:deta}, we have that if any of $\eta_I,\eta_J,\eta_K$ is zero, then all of them must vanish on $M$. Indeed, if for instance $\eta_I=0$ then by \eqref{eqn:deta}
\[
0=d\eta_I=- \eta_J \wedge KV^\flat + \eta_K \wedge JV^\flat,
\]
which implies that $\eta_J= \eta_K=0$. \par
If $\eta_L$ are zero, then $\beta_L= \frac{1}{2} \omega_L^T$ by \eqref{eqn:betas} and, furthermore, $\nabla^B \beta_L=0$. \par
Using the expression of the Bismut torsion in \eqref{eqn:tor2}, we get that for any vector field $X$
\begin{equation}\label{eqn:par}
\nabla^B_X H=( \nabla^B_X dV^\flat) \wedge V^\flat.
\end{equation}
In order to prove that $ \nabla^B_X H$ is zero, we only need to prove that $\nabla^B_X dV^\flat (Y,Z)$ vanishes for any $X \in \Gamma(\operatorname{TM})$ and $Y,Z$ vector fields different from (a multiple of ) $V$. \par
We observe that if we take $Y,Z \in \{ IV,JV,KV\}$ and $X$ any vector field, then the statement is trivially true, since $dV^\flat$ is horizontal and $IV,JV,KV$ are Bismut parallel. The same holds for $Y \in \{ IV,JV,KV\}$ and $Z$ horizontal. It remains to check the condition for $Y,Z$ both horizontal. If $X\in \{ V, IV,JV,KV\}$ is vertical, then 
\begin{equation} \label{eqn:nabla}
\begin{split}
 \nabla^B_X dV^\flat (Y,Z)=&X(dV^\flat(Y,Z)-dV^\flat(\nabla^B_X Y,Z)-dV^\flat(Y,\nabla^B_X Z)\\
 =& - Xg([Y,Z],V)-dV^\flat([X ,Y]+T^B(X,Y),Z)\\
 &-dV^\flat(Y,[X Z]+T^B(X,Z))\\
 =& -g([X,[Y,Z]]+[[X,Y],Z]+[Y,[X,Y]],V)\\
 & -dV^\flat(T^B(X,Y),Z)-dV^\flat(Y,T^B(X,Z))\\
 =& -dV^\flat((\iota_Y dX^\flat)^\sharp, Z)-dV^\flat(Y,(\iota_Z dX^\flat)^\sharp).
 \end{split}
\end{equation}
We are going to show that this is zero in any point $p$ of $M$. Let us fix a orthonormal basis $\{e_1,\dots,e_4\}$ of horizontal vectors in $\operatorname{T_pM}$ adapted to $I,J,K$. Then, denoted by $\{e^1,\dots,e^4\}$ the dual basis $$\omega_I^T=e^{12}+e^{34},  \ \omega_J^T=e^{13}+e^{42}, \  \omega_K^T=e^{14}+e^{23}.$$ Then we may write $Y$ and $Z$ at $p$ as $Y=Y^i e_i$ and $Z=Z^j e_j$. The equality \eqref{eqn:nabla} at $p$ yields
\begin{equation} \label{eqn:aaa}
 \nabla^B_X dV^\flat (Y,Z) (p)= Y^i Z^j \left[(dX^\flat)_{ik} (dV^\flat)_{jk}-(dX^\flat)_{jk} (dV^\flat)_{ik}\right].
\end{equation}
If $X=V$ then the above expression is clearly zero. For $X=IV$ we use the following fact: since $dV^\flat$ is horizontal and anti self dual we may write $dV^\flat$ at $p$ as $\lambda_1( e^{12}-e^{34})+ \lambda_2 (e^{13}-e^{42})+ \lambda_3 (e^{14}-e^{23})$, with $\lambda_1,\lambda_2,\lambda_3 \in \R$ (see the proof of Theorem 1\ref{thm:main}). The expression \eqref{eqn:aaa} is hence given by 
\[
 \nabla^B_{IV} dV^\flat (Y,Z) (p)= \frac{1}{2}Y^i Z^j [(\omega_I^T)_{ik} (dV^\flat)_{jk}-(\omega_I^T)_{jk} (dV^\flat)_{ik}],
\]
as $dIV^\flat(Y,Z)=\beta_I (Y,Z)= \frac{1}{2} \omega_I^T (Y,Z)$. Since $dV^\flat$ is $\mathrm{SU}(2)$-invariant,  $(\omega_I^T)_{ik} (dV^\flat)_{jk}-(\omega_I^T)_{jk} (dV^\flat)_{ik}=0$. The other cases follow in analogous way. \par
It remains to check that  $\nabla^B_X dV^\flat (Y,Z)=0$ for any $X,Y,Z$ horizontal vector fields. In this case 
\[
\nabla^B_X dV^\flat (Y,Z)=  \nabla^{T}_X dV^\flat (Y,Z),
\]
where $\nabla^T$ is the transverse Levi-Civita and the equality follows since for any $X,Y,Z$ horizontal vector fields $g(\nabla^{B}_X Y, Z)=g(\nabla^{LC}_X Y, Z)+ \frac{1}{2} H(X,Y,Z)=g(\nabla^{T}_X Y, Z)$. \par	
We observe that by the condition $dH=0$ \eqref{eqn:dH3}, we get that $dV^\flat$ has constant non zero norm. \par
We are going to define a transverse almost complex structure $\mathbb{J}^T$ on the bundle $\cal{F}^\perp$ using the metric $g^T$ and the basic $2$-form $dV^\flat$\footnote{By Remark \ref{rmk:useful} it is not difficult to check that $\cal{L}_{V} dV^\flat=\cal{L}_{IV} dV^\flat=\cal{L}_{JV} dV^\flat=\cal{L}_{KV} dV^\flat=0$.}. In particular, since $\cal{L}_{V} g^T=\cal{L}_{IV} g^T=\cal{L}_{JV} g^T=\cal{L}_{KV} g^T=0$, the pair $(\mathbb{J}^T, g^T)$ will induce an almost Hermitian structure on any local quotient of $M$ by the flows of $V,IV,JV,KV$. \par
Let $\mathbb{J}^T$ be the endomorphism defined on horizontal vector fields by the relation
\[
g^T(\mathbb{J}^TX,Y)= \frac{1}{\norm{dV^\flat}}dV^\flat (X,Y).
\]
 $\mathbb{J}^T$ defines an almost complex structure: this follows from that fact that $dV^\flat$ is a ASD horizontal form and the horizontal distribution is four dimensional.
We point out that the basic Hermitian structure $(\mathbb{J}^T, g^T)$ is only almost K\"ahler. Moreover we adapt the proof of \cite{ALS} in our case, which relay on some results in \cite{ADM}. \par
We first observe that the transverse metric is Einstein. By O'Neill formulas \cite{ON} (see also \cite[Section 2.5.1]{BG}), 
\[
Ric^T(X,Y)=Ric^{LC}(X,Y)+ 2 \big[ g(A_X V, A_Y V) +g(A_X IV, A_Y IV) +g(A_X JV, A_Y JV) +g(A_X KV, A_Y KV) \big],
\]
where $A_X V= \operatorname{pr} (\nabla^{LC}_X V)$ and $\operatorname{pr} : \operatorname{TM} \to \cal{F}^\perp$ is the standard projection. \par
Since by Theorem \ref{thm:main} $\nabla^B V= 2 \nabla^B \theta=0$, by formula \eqref{eqn:rhob2} $Ric^B=0$. In particular, its symmetric part has to vanish, implying that $Ric^{LC}=\frac{1}{4} H^2$, where $H^2(X,Y)=g(\iota_X H, \iota_Y H)$. On the other hand, for any $X, Z$ horizontal vector fields
\[
g(A_X V, Z)=g(\nabla^{LC}_X V, Z)=g(\nabla^{B}_X V, Z)- \frac{1}{2} H(X,V,Z)=\frac{1}{2} dV^\flat(X,Z),
\]
implying that $A_X V= \frac{1}{2} (\iota_XdV^\flat)^\sharp$. Similarly, $A_X IV= \frac{1}{4} (\iota_X\omega_I^T)^\sharp$, $A_X JV= \frac{1}{4} (\iota_X\omega_J^T)^\sharp$ and  $A_X KV= \frac{1}{4} (\iota_X\omega_K^T)^\sharp$. In particular, 
\[
Ric^T(X,Y)=\frac{3}{4} \big[ (dV^\flat)^2 (X,Y) +\frac{1}{4} (\omega_I^T)^2 (X,Y)+\frac{1}{4} (\omega_J^T)^2 (X,Y)+\frac{1}{4} (\omega_K^T)^2 (X,Y)   \big],
\]
where again $\alpha^2(X,Y)=g(\iota_X \alpha, \iota_Y \alpha)$. \par
We fix a local orthonormal horizontal basis $\{e_1,\dots, e_4\}$ of $\cal{F}^\perp$ adapted to $I,J,K$ in a open set $\cal{U}$. With respect to the dual basis $\{e^1,\dots,e^4\}$ we may write $\omega_I=e^{12}+e^{34}, \omega_J=e^{13}+e^{42}, \omega_K=e^{14}+e^{23}$ and $dV^\flat=\lambda_1( e^{12}-e^{34})+ \lambda_2 (e^{13}-e^{42})+ \lambda_3 (e^{14}-e^{23})$, with $\lambda_i$ smooth function defined in a neighborhood of $p$.
It is a straightforward computation that for any $q \in \cal{U}$ 
\[
Ric_q^T(e_i,e_j)= \frac{3}{4} \left(\frac{3}{4}+(\lambda_1^2+\lambda_2^2+\lambda_3^2)(q)\right) \delta^i_j.
\]
Using that $\lambda_1^2+\lambda_2^2+\lambda_3^2$ is a constant multiple of $\norm{dV}^2$ we get that the transverse metric $g_T$ is Einstein with strictly positive Einstein constant. \par
Now we argue as in \cite{ALS}. In fact, as pointed out in \cite{ALS}, the integral identity in \cite[Corollary 2.2]{ADM} also extends in our situation, in fact, also in our case it holds that the $\operatorname{L}^2$ product on $(M,g)$ satisfies $< \delta_g \pi^* \alpha, \pi^* \beta >=< \pi^* \delta_{g^{T}} \alpha, \pi^* \beta >$ for any $\alpha, \beta$ transversal forms. In particular, using that the transverse metric is Einstein with positive Einstein constant, we get that $\nabla^T dV^\flat=0$ (see \cite[Remark 2.3]{ADM} for further details). \par
Therefore, by \eqref{eqn:par}, $\nabla^B H=0$ and the result follows by Theorem \ref{thm:splitting2}.
\end{proof}
We are able now to fully characterize the condition $\nabla^B H$ for compact and simply connected $8$-dimensional strong HKT manifolds. 
\begin{theorem}
Let $(M^8,g,I,J,K)$ be a compact simply connected strong HKT manifold which is not hyper-K\"ahler. The following are equivalent:
\begin{enumerate} [label=\arabic*.]
\item $\nabla^B H=0$, i.e., $(M^8,g,I,J,K)$ is the hyper-Hermitian Samelson space $SU(3)$;
\item for any $L \in \{I,J,K\}$, $R^{LC}(V,LV,X,Y)=0$ for any $X,Y$ horizontal vector fields;
\item there exists $L \in \{I,J,K\}$ such that $R^{LC}(V,LV,X,Y)=0$ for any $X,Y$ horizontal vector fields;
\item for any $L \in \{I,J,K\}$, $\eta_L=0$;
\item there exists $L \in \{I,J,K\}$ such that $\eta_L=0$.
\end{enumerate}
\end{theorem} 
\begin{alignedproof}{4\gives5}
\alignedproofstep{4\gives5} Obvious. 
\alignedproofstep{5\gives4} It is a straightforward consequence of equation \eqref{eqn:deta}.
\alignedproofstep{5\gives1} It is Theorem \ref{thm:eta}.
\alignedproofstep{1\gives2} If $\nabla^B H=0$, and so $(M^8,g,I,J,K)$ is the hyper-Hermitian Samelson space $SU(3)$, then $(M^8,g,I,J,K)$ is Bismut flat. In particular, fixed a local horizontal orthonormal frame $\{e_1,\dots,e_4\}$, we have that for any $L \in \{I,J,K\}$, $R^{B}(V,LV,e_i,e_j)=0$. By Lemma \ref{lemma:appendix}, 
\[
R^{LC}(V,LV,e_i,e_j)= \frac{1}{4} R^{B}(V,LV,e_i,e_j)=0. 
\]
\alignedproofstep{2\gives3} Obvious.
\alignedproofstep{3\gives4} To lighten the notation we may assume that $L=I$, the other proofs are analogues. \footnote{The strategy of the proof is inspired by an argument appearing in a preliminary version of the work \cite{ALS}.}
Firstly we have to show that fixed any local horizontal orthonormal frame $\{e_1,\dots,e_4\}$ adapted to $I,J,K$ we have that for any $k=1,\dots,4$, 
\begin{equation} \label{eqn:identity}
\sum_{1 \le i < j \le 4} dV^\flat (e_i,e_j) (\nabla^{LC}_{e_{k}} \eta_I)(e_i,e_j)-\eta_I (e_i,e_j) (\nabla^{LC}_{e_{k}} dV^\flat)(e_i,e_j)=0.
\end{equation}
We prove this equality as a consequence of the second Bianchi identity
\begin{equation} \label{eqn:secondbianchi}
\nabla^{LC}_{e_{i}} R^{LC}(V,IV,{e_{l}},{e_{j}})+\nabla^{LC}_{e_{l}} R^{LC}(V,IV,{e_{j}},{e_{i}})+\nabla^{LC}_{e_{j}} R^{LC}(V,IV,{e_{i}},{e_{l}})=0.
\end{equation}
We study in details the first term $\nabla^{LC}_{e_{i}} R^{LC}(V,IV,{e_{l}},{e_{j}})$, as the other ones behave similarly. \\
We observe that since $R^{LC}(V,IV,{e_{l}},{e_{j}})=0$ by hypothesis, 
\[
\begin{split}
\nabla^{LC}_{e_{i}} R^{LC}(V,IV,{e_{l}},{e_{j}})=&-R^{LC}(\nabla^{LC}_{e_{i}} V,IV,{e_{l}},{e_{j}})-R^{LC}(V,\nabla^{LC}_{e_{i}} IV,{e_{l}},{e_{j}})\\
&-R^{LC}(V,IV,\nabla^{LC}_{e_{i}} {e_{l}},{e_{j}})-R^{LC}(V,IV,{e_{l}},\nabla^{LC}_{e_{i}} {e_{j}}).
\end{split}
\]
We examine each term of the above expression. \par
Since $e_i$ is horizontal, for any vertical vector field $U$,  $g(\nabla^{LC}_{e_{i}} V, U)=g(\nabla^B_{e_{i}} V, U)-\frac{1}{2} H({e_{i}},V,U)=0$, as $V$ is Bismut parallel and by the expression of the torsion $H$ given in \eqref{eqn:tor2}. On the other hand, if $U$ is instead horizontal, $g(\nabla^{LC}_{e_{i}} V, U)=g(\nabla^B_{e_{i}} V, U)-\frac{1}{2} H({e_{i}},V,U)=\frac{1}{2} dV^\flat({e_{i}},U)$, implying that $\nabla^{LC}_{e_{i}} V=\frac{1}{2}(\iota_{e_{i}} dV^\flat)^\sharp$. Hence, 
\[
\ \ \ \ \ \ \ \ \ \ \ \ R^{LC}(\nabla^{LC}_{e_{i}}V,IV,{e_{l}},{e_{j}})=\frac{1}{2}R^{LC}((\iota_{e_{i}} dV^\flat)^\sharp,IV,{e_{l}},{e_{j}})=\frac{1}{2} \sum_{k=1}^4 dV^\flat(e_i,e_k)R^{LC}(e_k,IV,{e_{l}},{e_{j}}).
\]
The same argument applies for $R^{LC}(V,\nabla^{LC}_{e_{i}} IV,{e_{l}},{e_{j}})$. \\
For the term $R^{LC}(V,IV,\nabla^{LC}_{e_{i}}{e_{l}},{e_{j}})$ we may argue in this way. By hypothesis, $$R^{LC}(V,IV,\nabla^{LC}_{e_{i}}{e_{l}},{e_{j}})=R^{LC}(V,IV,\pi_V(\nabla^{LC}_{e_{i}}{e_{l}}),{e_{j}}),$$ where $\pi_V$ is the vertical projection. Moreover, for any vertical vector $U$, $R^{LC}(V,IV,U,Y)=0$, again by using Lemma \ref{lemma:appendix}. Therefore
\[
\ \ \ \ \ \ \ \ \ \ \ \ R^{LC}(V,IV,\nabla^{LC}_{e_{i}}{e_{l}},{e_{j}})=0.
\]
The last term can be done similarly.\\
To summarize, using Lemma \ref{lemma:appendix} we have that
\[
\begin{split}
\ \ \	\ \ \ \ \ \ \ \  \ \ \  \nabla^{LC}_{e_{i}} R^{LC}(V,IV,{e_{l}},{e_{j}})=&\frac{1}{2} \sum_{k=1}^4 dV^\flat(e_i,e_k)R^{LC}(IV,e_k,e_{l},e_{j})- (\beta_I) (e_i,e_k) R^{LC}(V,e_k,{e_{l}},{e_{j}})\\
=& \frac{1}{4} \sum_{k=1}^4 -dV^\flat(e_i,e_k)\nabla^{LC}_{e_{k}}\eta_I (e_{l},e_{j})+ (\beta_I) (e_i,e_k) \nabla^{LC}_{e_{k}}dV^\flat(e_{l},{e_{j}}).
\end{split}
\]
Dealing with the other terms in a similar fashion, we have that the second Bianchi identity \eqref{eqn:secondbianchi} reads as
\begin{equation}\label{eqn:components}
\begin{split}
\ \ \ \ \ \ \ \ 0=&  \sum_{k=1}^4-(dV^\flat)_{ik} (\nabla^{LC}_{e_{k}} \eta_I)_{lj}
-(dV^\flat)_{lk} (\nabla^{LC}_{e_{k}}\eta_I)_{ji}
-(dV^\flat)_{jk} (\nabla^{LC}_{e_{k}} \eta_I)_{il} \\
&+(\beta_{I})_{ik} (\nabla^{LC}_{e_{k}}dV^\flat)_{lj}
+(\beta_{I})_{lk} (\nabla^{LC}_{e_{k}}dV^\flat)_{ji}
+(\beta_{I})_{jk} (\nabla^{LC}_{e_{k}}dV^\flat)_{il},
\end{split}
\end{equation}
where we omitted the constant factor $\frac{1}{4}$. \par
We also point out that since $R^{B}(V,X,Y,Z)=R^{B}(V,X,IY,IZ)=R^{B}(V,X,JY,JZ)=R^{B}(V,X,KY,KZ)$ (this holds in general for any HKT manifold since $R^B$ has values in $(1,1)$ forms), we have the following identities
\[
\begin{split}
&\nabla^{LC}_X dV^\flat (e_1,e_2)=-\nabla^{LC}_X dV^\flat (e_3,e_4), \ \nabla^{LC}_X \eta_I (e_1,e_2)=-\nabla^{LC}_X \eta_I(e_3,e_4) \\
&\nabla^{LC}_X dV^\flat (e_1,e_3)= \ \ \nabla^{LC}_X dV^\flat (e_2,e_4), \ \nabla^{LC}_X \eta_I (e_1,e_3)=\ \  \ \nabla^{LC}_X \eta_I(e_2,e_4) \\
&\nabla^{LC}_X dV^\flat (e_1,e_4)=-\nabla^{LC}_X dV^\flat (e_2,e_3), \ \nabla^{LC}_X \eta_I (e_1,e_4)=-\nabla^{LC}_X \eta_I(e_2,e_3). \\
\end{split}
\]
We start by taking $i=1, l=2, j=3$ in \eqref{eqn:components}. The first line of \eqref{eqn:components} can be written as 
\begin{equation} \label{eqn:123}
\begin{split}
&-dV^\flat_{12}[ (\nabla^{LC}_2 \eta_I)_{23}+(\nabla^{LC}_1 \eta_I)_{13}-(\nabla^{LC}_4 \eta_I)_{12}] -dV^\flat_{13}[ (-\nabla^{LC}_3 \eta_I)_{32}-(\nabla^{LC}_4 \eta_I)_{13}\\
&-(\nabla^{LC}_1 \eta_I)_{12}]  -dV^\flat_{14}[ (-\nabla^{LC}_4 \eta_I)_{14}-(\nabla^{LC}_3 \eta_I)_{31}-(\nabla^{LC}_2 \eta_I)_{21}]. \\
\end{split}
\end{equation}
Using that for any $j=1,\dots,4$, the codifferential $\delta \eta_I (e_j)=0$ and that for any $U \in \{ V,IV,JV,KV\}$ we have $\nabla^{LC}_U \eta_I (U,e_j)=-\eta_I (\nabla^{LC} _U U ,e_j)=0$, we get that
\[
\ \ \ \ \ \ \ \ \ \delta \eta_I (e_j)=- \sum_{i=1}^4 \nabla^{LC}_{e_{i}} \eta_I(e_i,e_j) =0.
\]
In particular, \eqref{eqn:123} can be written also as
\begin{equation*} 
\begin{split}
&-dV^\flat_{12}[ (-\nabla^{LC}_4 \eta_I)_{43}-(\nabla^{LC}_4 \eta_I)_{12}] -dV^\flat_{13}[ (\nabla^{LC}_4 \eta_I)_{42}-(\nabla^{LC}_4 \eta_I)_{13}]\\
& -dV^\flat_{14}[ (-\nabla^{LC}_4 \eta_I)_{14}+(\nabla^{LC}_4 \eta_I)_{41}]  \\
&=-dV^\flat_{12}[ -2(\nabla^{LC}_4 \eta_I)_{12}] -dV^\flat_{13}[ -2(\nabla^{LC}_4 \eta_I)_{13}]-dV^\flat_{14}[-2 (\nabla^{LC}_4 \eta_I)_{14}] = \\
&= 2[dV^\flat_{12}(\nabla^{LC}_4 \eta_I)_{12} +dV^\flat_{13}(\nabla^{LC}_4 \eta_I)_{13}+dV^\flat_{14} (\nabla^{LC}_4 \eta_I)_{14}]=\\
&=2\sum_{1 \le i < j \le 4} dV^\flat_{ij} (\nabla^{LC}_4 \eta_I)_{ij}. \\
\end{split}
\end{equation*}
Analogously, the second line of \eqref{eqn:components} is
 \begin{equation*} 
 \begin{split}
& \ \ \ \ \eta_{I_{12}}[ (\nabla^{LC}_2 dV^\flat)_{23}+(\nabla^{LC}_1dV^\flat)_{13}-(\nabla^{LC}_4dV^\flat)_{12}] +\eta_{I_{13}}[ (-\nabla^{LC}_3 dV^\flat)_{32}-(\nabla^{LC}_4 dV^\flat)_{13}\\
&   -(\nabla^{LC}_1 dV^\flat)_{12}]  +\eta_{I_{14}}[ (-\nabla^{LC}_4dV^\flat)_{14}-(\nabla^{LC}_3 dV^\flat)_{31}-(\nabla^{LC}_2 dV^\flat)_{21}]  \\
&    +\frac{1}{2} [(\nabla^{LC}_2 dV^\flat)_{34}+ (\nabla^{LC}_4 dV^\flat_I)_{41}-(\nabla^{LC}_3 dV^\flat)_{13} ]\\
\ \  &= \eta_{I_{12}}[ (\nabla^{LC}_2 dV^\flat)_{23}+(\nabla^{LC}_1dV^\flat)_{13}-(\nabla^{LC}_4dV^\flat)_{12}] +\eta_{I_{13}}[ (-\nabla^{LC}_3 dV^\flat)_{32}-(\nabla^{LC}_4 dV^\flat)_{13}\\
&  -(\nabla^{LC}_1 dV^\flat)_{12}]  +\eta_{I_{14}}[ (-\nabla^{LC}_4dV^\flat)_{14}-(\nabla^{LC}_3 dV^\flat)_{31}-(\nabla^{LC}_2 dV^\flat)_{21}] \\
&  +\frac{1}{2} [(\nabla^{LC}_2 dV^\flat)_{34}+ (\nabla^{LC}_2 dV^\flat)_{12}] \\
\ \ &= -2\sum_{1 \le i < j \le 4} \eta_{I_{ij}} (\nabla^{LC}_4 dV^\flat)_{ij},
\end{split}
\end{equation*}
where we used again that $\delta dV^\flat (e_j)=0$, and so $\delta dV^\flat_I (e_j)=- \sum_{i=1}^4 \nabla^{LC}_{e_{i}} dV^\flat_I(e_i,e_j) =0$. \\
Hence, for  $i=1, j=2, l=3$ in \eqref{eqn:components} we get that
\[
\frac{1}{2}\sum_{1 \le i < j \le 4} dV^\flat_{ij} (\nabla^{LC}_4 \eta_I)_{ij}-\frac{1}{2} \sum_{1 \le i < j \le 4} \eta_{I_{ij}} (\nabla^{LC}_4 dV^\flat)_{ij}=0.
\]
The same result for the other values of $k$ follows by analogues computations. \par
Having showed equation \eqref{eqn:identity}, we may argue as follows. With respect to the dual frame $\{e^1,\dots,e^4\}$ we may write $$
dV^\flat=a_1(e^{12}-e^{34})+a_2(e^{13}-e^{42})+a_3(e^{14}-e^{23}),$$ and $$ \eta_I=b_1(e^{12}-e^{34})+b_2(e^{13}-e^{42})+b_3(e^{14}-e^{23}),$$
where $a_i$ and $b_i$ are smooth functions. Applying Lemma \ref{lemma:appendix}, 
\begin{equation}\label{eqn:lccurvature}
\begin{split}
0=&R^{LC}(V,LV,e_1,e_2)=2(b_2a_3-b_3a_2), \\
0=&R^{LC}(V,LV,e_1,e_3)=2(b_3a_1-b_1a_3), \\ 
0=&R^{LC}(V,LV,e_1,e_4)=2(b_1a_2-b_2a_1).
\end{split}
\end{equation}
By Proposition \ref{prop:dtheta}, we have that $dV^\flat$ has only isolated zeros. Let $Z:=\{p \in M \ | \ dV^\flat (p)=0\}$ and let $p$ be any point of $M \setminus Z$. Then, either $\eta_I (p)=0$ or $\eta_I(p) \neq 0$. If $\eta_I(p) \neq 0$ there exists an open neighborhood $\cal{U}$ of $p$ such that $\eta_I$ is never vanishing on $\cal{U}$, and if we take $\cal{V}:=\cal{U} \cap M\setminus Z$, then $\cal{V}$ is an open neighborhood of $p$ on which both $\eta_I$ and $dV^\flat$ are never vanishing. \par
We prove that the second case cannot occur, i.e., that $\eta_I$ is zero on $M \setminus Z$.\par
Assume by contradiction that $\eta_I(p) \neq 0$. By looking more carefully at \eqref{eqn:lccurvature}, we get that on $\cal{V}$ there exists a smooth function $\lambda$ such that $dV^\flat=\lambda \eta_I$, with $\lambda$ being never vanishing, otherwise $dV^\flat$ would have a zero on $\cal{V}$. \par
On $\cal{V}$ we apply \eqref{eqn:identity} and we get that for any $k=1,\dots,4$,  
\[
\begin{split}
0=&\sum_{1 \le i < j \le 4} \lambda (\eta_{I})_{ij} (\nabla^{LC}_{e_{k}} \eta_I)_{ij}-(\eta_{I})_{ij} (\nabla^{LC}_{e_{k}} \lambda \eta_I)_{ij}\\
=&\sum_{1 \le i < j \le 4} \lambda (\eta_{I})_{ij} (\nabla^{LC}_{e_{k}} \eta_I)_{ij}-(\eta_{I})_{ij}{e_{k}} (\lambda) (\eta_{I})_{ij}-\lambda (\eta_{I})_{ij} (\nabla^{LC}_{e_{k}} \eta_I)_{ij}\\
=&-\sum_{1 \le i < j \le 4}((\eta_{I})_{ij})^2{e_{k}} (\lambda) \\
=&- \norm{\eta_I}^2 {e_{k}} (\lambda),
\end{split}
\]
which forces $\lambda$ to be constant on $\cal{V}$ in the horizontal directions. In particular
\[
d\lambda=\lambda_1 V^\flat + \lambda_2 IV^\flat + \lambda_3 JV^\flat + \lambda_4 KV^\flat,
\]
where $\lambda_i$ are smooth functions. Differentiating again we get 
\[
\begin{split}
0=&d\lambda_1 \wedge V^\flat + \lambda_1 dV^\flat + d\lambda_2\wedge IV^\flat + \lambda_2 JV^\flat \wedge KV^\flat + \lambda_2 \beta_I +  d\lambda_3\wedge JV^\flat \\ & + \lambda_3 KV^\flat \wedge IV^\flat + \lambda_3 \beta_J + d\lambda_4\wedge KV^\flat + \lambda_4 IV^\flat \wedge JV^\flat + \lambda_4 \beta_K, 
\end{split}
\]
which forces 
\[
 \lambda_1 dV^\flat + \lambda_2 \beta_I + \lambda_3 \beta_J + \lambda_4 \beta_K=0.
\]
In particular, the  self dual  part and the anti self dual  parts of the above expression have to be zero, which means that
\[
\lambda_2 \omega_I^T + \lambda_3 \omega_J^T + \lambda_4 \omega_K^T=0 \ \text{and} \ \lambda_1 dV^\flat+ \lambda_2 \eta_I + \lambda_3 \eta_J + \lambda_4 \eta_K=0,
\]
i.e., $\lambda_1=\lambda_2=\lambda_3=\lambda_4=0$, since $dV^\flat$ is never vanishing on $\cal{V}$. In particular, $\lambda$ is constant and non zero. By differentiating $dV^\flat=\lambda \eta_I$, we get
\[
0=\lambda  d\eta_I= \lambda (- \eta_J \wedge KV^\flat + \eta_K \wedge JV^\flat),
\]
which forces $\lambda=0$, a contradiction. In fact, since $\eta_I$ is non zero on $\cal{V}$, so are $\eta_J$ and $\eta_K$. \par
Therefore, $\eta_I$ must vanish on $M \setminus Z$ and so it must vanish on $M$, as $Z$ is a set of isolated points (see Proposition \ref{prop:dtheta}) and $\eta_I$ is smooth. This concludes the proof.
\end{alignedproof}

\smallskip

\section{HKT potential and conical hypercomplex structures for strong HKT manifolds}\label{Potential}
Existence of rotational vector fields on hyper-K\"ahler manifolds appeared first in Swann bundle construction \cite{Swann}. 
For a hyper-K\"ahler manifold $(M, I,J,K,g)$ with fundamental forms $\omega_i, \omega_J,\omega_K$ it was shown there that the existence of a common local potential function $\mu$ such that $\omega_I=dd^I\mu, \omega_J=dd^J\mu, \omega_K = dd^K\mu$, leads to a Killing vector field $V_{\mu}$ dual to $d\mu$ and $\mu = \frac{1}{2}||V||^2$, such that $IV_{\mu}, JV_{\mu}, KV_{\mu}$ rotate the complex structures and satisfy conditions 1) and 2) of Corollary \ref{cor:properties} above. In particular the foliation in \cite{Swann} is a fibration over a quaternionic-K\"ahler space. 

One advantage of the HKT geometry is the general existence of a local HKT potential. This is a local function $\mu$ on a hyper-Hermitian manifold such that 
\[
\omega_I = (dd^I+d^Jd^K)\mu, \omega_J = (dd^J+d^Kd^I)\mu, \omega_K = (dd^K+d^Id^J)\mu
\]
It is easy to see that the metric defined by such function is HKT and its torsion is expressed as $H=d^Id^Jd^K\mu$. In particular the strong HKT condition is $dd^Id^Jd^K\mu = 0$. Equivalently a potential function is defined as $\Omega = \omega_J+i\omega_K = \partial_I J\overline{\partial_I}\mu$ 
and two different HKT potentials differ by a quaternionic pluriharmonic function $p$, 
 i.e. $\partial_I J\overline{\partial_I}p=0$.   A local HKT potential is also called quaternionic plurisubharmonic function. More details  can be found in \cite{Gr-Poon}, \cite{Ve} and the general local existence is proved in \cite{Ba-Swann}. 
 
The Swann bundle construction was extended in \cite{Poon-Swann} to the HKT case  where the HKT potential of a manifold with $D(2,1,\alpha)$-symmetry is given by the norm of a special vector field $V$ generating the symmetry. Such vector field is called the Euler vector field. It again gives rise of a local $\mathbb{H}^*$ action by the quaternionic span of $V$. 

As mentioned above, for the Swann bundle over hyper-K\"ahler manifold, the K\"ahler potential for all forms is determined by the norm of the Euler vector field, 
while in our case, the norm of $V$ is constant and the relation to the potential is not that direct. But by Theorem \ref{thm:main}, we also conclude that, up to constant multiples, the vector field $IV$ is rotational \cite{CH2}. Rotational vector fields appear in the HK/QK correspondence of A. Haydys \cite{H} which is extended to  hypercomplex/quaternionic correspondence in \cite{CH2}.  To a hypercomplex manifold equipped with a rotational vector field and additional data the correspondence associates a {\it conical} quaternionic manifold of the same dimension. In particular the manifolds in Theorem \ref{thm:main} are conical, following the terminology of \cite{CH2}. 

Although the relation with the HKT potential is not as straightforward as in \cite{Poon-Swann}, the  statement $3.$ in Theorem  \ref{thm:main}  provides a relation to the notion of potential 1-form (see \cite{OPS}). Let $\Omega = \omega_J+i\omega_K$. From the equations above we have $\cal{L}_{IV}\Omega= i\Omega$, and, for the holomorphic vector field $V^{1,0}=V-iIV$, we get by linearity $\cal{L}_{V^{1,0}}\Omega = \Omega$.  Using the Cartan formula for the $(2,0)$-parts of $\cal{L}_{V^{1,0}}\Omega = \Omega$, we have $\partial \iota_{V^{1,0}}\Omega = \Omega$ and, analogously, the Cartan formula for the $(1,1)$-part gives $ \overline{\partial}_I\iota_{V^{1,0}} \Omega+ \iota_{V^{1,0}} \overline{\partial}_I \Omega=0$.\\
From $\iota_{V^{1,0}}\Omega = 2(JV^{\flat})^{1,0}$ we get that 
$$\omega_J = d(JV^{\flat})-I d(JV^{\flat}),$$ which  follows also from the formulas above. This is also the definition of $V^{\flat}$ being a {\it potential 1-form} in \cite[Definition 4]{OPS}.  However, $V^{\flat}$ is not closed. We further observe that $\Omega$ is $\partial_I$-exact,   as $\Omega=\partial_I \iota_{V^{1,0}}\Omega=\partial_I (2(JV^{\flat})^{1,0})$. We point out that when the canonical bundle is trivial, the HKT form $\Omega$ cannot be exact \cite{Ve}. \par
On the other hand, for a local potential function $\mu$, determined up to a quaternionic plurisubharmonic function, we have $$\Omega=\partial_I J\overline{\partial_I} \mu.$$ From here it follows that $\partial_I(J\overline{\partial_I}\mu - (JV^{\flat})^{1,0}_I) = 0$
and for some local function $p$, $$(V^{\flat})^{0,1}_I = J\partial_I p+\overline{\partial_I}\mu.$$ \\
We also point out that the $(1,0)$ form $(V^\flat)^{1,0}_I$ given by $V^\flat+ i IV^\flat$ is $\partial_I$-closed. In fact, $d(V^\flat+ i IV^\flat)$ is of type $(1,1)$. Analogously, the $(0,1)$ form $(V^\flat)^{0,1}_I$ given by $V^\flat- i IV^\flat$ is $\overline{\partial_I}$-closed, so, in particular, $\overline{\partial}J\partial p=0$ is a quaternionic plurisubharmonic function.
Comparing the real parts in $(V^{\flat})^{0,1}_I = J\partial_I p+\overline{\partial_I}\mu$ the observations lead to the following:

\begin{proposition} \label{Prop-potential}
For a compact simply connected 8-dimensional strong HKT manifold $M$ with a vector field $V$ as above and a local HKT potential function $\mu$ there exists a local quaternionic pluriharmonic function $p$, such that  \begin{equation} \label{formula-potential} V^{\flat} = Jdp+d\mu.\end{equation}

\end{proposition}

\begin{remark} We note that the local potential $\mu$  is determined up to a local quaternionic pluriharmonic  function.  In fact,  Proposition \ref{Prop-potential}  reveals that a  strong HKT  structure  in dimension $8$  depends  locally on two  functions $p$ and $\mu$,  with  $p$  determined by  $\mu$ up to a constant. We can see for example that from $$\omega_I = V^{\flat}\wedge IV^{\flat} + JV^{\flat}\wedge KV^{\flat} + \omega_I^T = (dd^I +d^Jd^K) \mu,$$ using    \eqref{formula-potential},  we  get  the following expression of $\omega_I^T$ in terms of $\mu$ and $p$
\begin{align*}
\omega_I^T  & = (dd^I +d^Jd^K) \mu - d \mu \wedge I d \mu - J d \mu \wedge K d \mu - d p \wedge I dp - J dp \wedge K dp\\[2pt]
& + d p \wedge K d \mu - d \mu \wedge K dp - J dp \wedge I d \mu + J d \mu \wedge I dp
\end{align*}
 and similar expressions are valid for $\omega_J^T$ and $\omega_K^T$. Similarly, from $$d(IV^{\flat})= JV^{\flat}\wedge KV^{\flat}+\frac 12 \omega_I^T+\eta_I,$$ we have an expression for $\eta_I$ in terms of $p$ and $\mu$, and correspondingly for $\eta_J$ and $\eta_K$.
\end{remark}

\appendix
\section{Curvature components}
In this Appendix, we compute explicitly some components of the Bismut and the Levi-Civita curvatures. 
\begin{lemma} \label{lemma:appendix}
Let $(M^8,g,I,J,K)$ be a compact simply connected strong HKT manifold which is not hyper-K\"ahler. Let $\cal{F}$ be the distribution generated by $V,IV,JV,KV$ and let $\cal{F}^\perp$ the orthogonal complement of $\cal{F}$ with respect to $g$. Then fixed any local horizontal orthonormal frame $\{e_1,\dots,e_4\}$ adapted to $I,J,K$ and any $L \in \{I,J,K\}$ we get
\begin{enumerate}[label=\arabic*.]
\item $R^{B}(V,LV,e_i,e_j)=dV^\flat((\iota_{e_{i}}\eta_L)^\sharp,e_j)+dV^\flat(e_i,(\iota_{e_{j}}\eta_L)^\sharp)$ 
\item $R^{B}(V,e_i,e_j,e_k)=  - (\nabla^{LC}_{e_i} dV^\flat) (e_j,e_k)$
\item $R^{B}(LV,e_i,e_j,e_k)=  - (\nabla^{LC}_{e_i} \eta_L) (e_j,e_k)$
\item $R^{LC}(V,LV,e_i,e_j)=\frac{1}{4} R^{B}(V,LV,e_i,e_j)$
\item $R^{LC}(V,e_i,e_j,e_k)=  \frac{1}{2} R^{B}(V,e_i,e_j,e_k)$
\item $R^{LC}(LV,e_i,e_j,e_k)=  \frac{1}{2} R^{B}(LV,e_i,e_j,e_k)$
\item $R^{LC}(V,LV,U,e_i)=  0$, for any $U$ vertical vector field.
\end{enumerate}
\end{lemma}
\begin{proof}
Since for any $L=I,J,K$ the Hermitian structure $(g,L)$ is BHE, by \cite[Lemma 2.13]{ALS} we have that for any $X,Y,Z \in \ker(V^\flat) \cap \ker( LV^\flat)$ 
\begin{equation}\label{eqn:curvature}
R^{B}(V,X,Y,Z)=-\nabla^{B}_X dV^\flat (Y,Z), \ \  R^{B}(LV,X,Y,Z)=-\nabla^{B}_X dLV^\flat (Y,Z)=-\nabla^{B}_X \eta_L (Y,Z).
\end{equation}
\begin{enumerate}[label=\arabic*.]
\item We prove the first statement for $L=I$, since the others follow analogously. If we take $X=IV,$ and $Y,Z$ horizontal vector fields, we have that $X,Y,Z$ lie for instance in $ \ker(V^\flat) \cap \ker( JV^\flat)$, so we may apply the first identity in \eqref{eqn:curvature}.
\[
\begin{split}
R^{B}(V,IV,Y,Z)=&- \nabla^B_{IV} dV^\flat (Y,Z)\\
=& dV^\flat((\iota_Y dIV^\flat)^\sharp, Z)+dV^\flat(Y,(\iota_Z dIV^\flat)^\sharp)\\
=&dV^\flat((\iota_Y \beta_I)^\sharp, Z)+dV^\flat(Y,(\iota_Z \beta_I)^\sharp)\\
=&dV^\flat((\iota_Y \eta_I)^\sharp, Z)+dV^\flat(Y,(\iota_Z \eta_I)^\sharp),
\end{split}
\]
where we used that $dV^\flat((\iota_Y \omega^T_I)^\sharp, Z)+dV^\flat(Y,(\iota_Z \omega_I^T)^\sharp)=0$ (observe that we already did this computation in details in the proof of Theorem \ref{thm:eta}). Therefore for $Y=e_i$ and $Z=e_j$ we get that
\[
R^{B}(V,IV,e_i,e_j)=dV^\flat((\iota_{e_{i}}\eta_I)^\sharp,e_j)+dV^\flat(e_i,(\iota_{e_{j}}\eta_I)^\sharp).
\]
\item Let $X,Y,Z$ be horizontal vector fields, which ie in $\ker(V^\flat) \cap \ker( LV^\flat)$, for any $L$. Then, by the first identity in \eqref{eqn:curvature}, $$R^{B}(V,X,Y,Z)=-\nabla^{B}_X dV^\flat (Y,Z)=-\nabla^{LC}_X dV^\flat (Y,Z),$$ as $H(X,Y,Z)=0$. By choosing $X=e_i$, $Y=e_j$ and $Z=e_k$ the statement follows. 
\item As in the proof of the first statement we deal with the case $L=I$, as the other computations are analogues. As before, let $X,Y,Z$ be horizontal vector fields, which ie in $\ker(V^\flat) \cap \ker( IV^\flat)$. Then by the second identity of \eqref{eqn:curvature} we have that $$R^{B}(IV,X,Y,Z)=-\nabla^{B}_X dIV^\flat (Y,Z)=-\nabla^{B}_X \eta_I (Y,Z) =-\nabla^{LC}_X \eta_I (Y,Z).$$ Also in this case we chose $X=e_i$, $Y=e_j$ and $Z=e_k$. 
\end{enumerate}
It only remains to verify the identities for the Levi-Civita curvature. We will use the \cite[Formula 3.19]{IP} which relates the Levi-Civita curvature with the Bismut one. In particular, for any  vector fields $X,Y,Z,U$ we have that
\begin{equation} \label{eqn:ivanovpap}
\begin{split}
R^{LC} (X,Y,Z,U)=& R^B (X,Y,Z,U) -\frac{1}{2} \nabla^B_X H (Y,Z,U) \\
+&\frac{1}{2} \nabla^B_Y H (X,Z,U) -\frac{1}{2} g(H(X,Y),H(Z,U))\\
-&\frac{1}{4} g(H(Y,Z),H(X,U))+\frac{1}{4} g(H(X,Z),H(Y,U)),
\end{split}
\end{equation}
where the expression of the torsion $H$ is given in \eqref{eqn:tor2}. 
\begin{enumerate}[label=\arabic*.] \addtocounter{enumi}{3}
\item Also in this case we deal with $L=I$. Applying \eqref{eqn:ivanovpap}:
\[
\begin{split}
\ \ \ \ \ R^{LC} (V,IV,e_i,e_j)=& R^B (V,IV,e_i,e_j) -\frac{1}{2} \nabla^B_{V} H (IV,e_i,e_j) +\frac{1}{2} \nabla^B_{IV} H (V,e_i,e_j)\\
-&\frac{1}{4} g(H(IV,e_i),H(V,e_j))+\frac{1}{4} g(H(V,e_i),H(IV,e_j)) \\
=&dV^\flat((\iota_{e_{i}}\eta_I)^\sharp,e_j)+dV^\flat(e_i,(\iota_{e_{j}}\eta_I)^\sharp) -\frac{1}{2} \nabla^B_{V} \eta_I(e_i,e_j) \\
+&\frac{1}{2} \nabla^B_{IV} dV^\flat(e_i,e_j)-\frac{1}{4} g((\iota_{e_{i}} \beta_I)^\sharp,(\iota_{e_{j}} dV^\flat)^\sharp)+\frac{1}{4} g((\iota_{e_{j}} \beta_I)^\sharp,(\iota_{e_{i}} dV^\flat)^\sharp) \\
=& dV^\flat((\iota_{e_{i}}\eta_I)^\sharp,e_j)+dV^\flat(e_i,(\iota_{e_{j}}\eta_I)^\sharp)- \frac{1}{2} dV^\flat((\iota_{e_{i}} \eta_I)^\sharp, e_j)\\
-&\frac{1}{2} dV^\flat(e_i,(\iota_{e_{j}} \eta_I)^\sharp)+\frac{1}{2} \eta_I((\iota_{e_{i}} dV^\flat)^\sharp, e_j)+\frac{1}{2}\eta_I(e_i,(\iota_{e_{j}} dV^\flat)^\sharp)\\
-&\frac{1}{4} g((\iota_{e_{i}} \eta_I)^\sharp,(\iota_{e_{j}} dV^\flat)^\sharp)+\frac{1}{4} g((\iota_{e_{j}} \eta_I)^\sharp,(\iota_{e_{i}} dV^\flat)^\sharp) \\
=& \frac{1}{4} (dV^\flat)_{ik} (\eta_I)_{jk} - \frac{1}{4} (dV^\flat)_{jk} (\eta_I)_{ik}\\
=& \frac{1}{4} R^B(V,IV,e_i,e_j). 
\end{split}
\] 
\item Exploiting the formula \eqref{eqn:ivanovpap},
\[
\begin{split}
\ \ \ \ \ R^{LC} (V,e_i,e_j,e_k)=& R^B (V,e_i,e_j,e_k) -\frac{1}{2} \nabla^B_{V} H (e_i,e_j,e_k) \\
+&\frac{1}{2} \nabla^B_{e_i} H (V,e_j,e_k)-\frac{1}{2} g(H(V,e_i),H(e_j,e_k))\\
-&\frac{1}{4} g(H(e_i,e_j),H(V,e_k))+\frac{1}{4} g(H(V,e_j),H(e_i,e_k)) \\
=&  - (\nabla^{LC}_{e_i} dV^\flat) (e_j,e_k) +\frac{1}{2} (\nabla^B_{e_i} dV^\flat)(e_j,e_k)\\
=&- \frac{1}{2} (\nabla^{LC}_{e_i} dV^\flat)(e_j,e_k)\\
=&  \frac{1}{2}R^B (V,e_i,e_j,e_k).
\end{split}
\] 
\item As in the previous cases we work with $L=I$ since the other cases follows similarly. By applying \eqref{eqn:ivanovpap},
\[
\begin{split}
\ \ \ \ \ R^{LC} (IV,e_i,e_j,e_k)=& R^B (IV,e_i,e_j,e_k) -\frac{1}{2} \nabla^B_{IV} H (e_i,e_j,e_k) \\
+&\frac{1}{2} \nabla^B_{e_i} H (IV,e_j,e_k)-\frac{1}{2} g(H(IV,e_i),H(e_j,e_k))\\
-&\frac{1}{4} g(H(e_i,e_j),H(IV,e_k))+\frac{1}{4} g(H(IV,e_j),H(e_i,e_k)) \\
=&  - (\nabla^{LC}_{e_i} \eta_I) (e_j,e_k) +\frac{1}{2} (\nabla^B_{e_i} \eta_I)(e_j,e_k)\\
=&- \frac{1}{2} (\nabla^{LC}_{e_i} \eta_I)(e_j,e_k)\\
=&  \frac{1}{2}R^B (IV,e_i,e_j,e_k).
\end{split}
\] 
\item We set again $L=I$. Without loss of generality, we may assume that $U \in \{V,IV,JV,KV\}$. In such a way, $\nabla^{B} U=0$, \ $R^{B} (V,IV,U,X)=0$ and $[U,V]=0$. Using \eqref{eqn:ivanovpap},
\[
\begin{split}
\ \ \ \ \ R^{LC} (V,IV,U,e_j)=& -\frac{1}{2} \nabla^B_{V} H (IV,U,e_j) +\frac{1}{2} \nabla^B_{IV} H (V,U,e_j)-\frac{1}{4} g(H(IV,U),H(V,e_j)) \\
=& -\frac{1}{2} \nabla^B_{V} \eta_I (U,e_j) +\frac{1}{2} \nabla^B_{IV} dV^\flat (U,e_j). \\
=& 0. 
\end{split}
\] 
\end{enumerate}
\end{proof}

\end{document}